\documentclass[12pt,a4paper]{amsart}

\usepackage[top=3cm, bottom=3cm, left=2.5cm, right=2.5cm]{geometry}

\usepackage{fullpage}

\usepackage{amssymb}
\usepackage{amsmath,bm}
\usepackage{amscd}
\usepackage[latin1]{inputenc}
\usepackage{wrapfig}
\usepackage{psfrag}
\usepackage{epsfig}
\usepackage{epic}
\usepackage{eepic}
\usepackage{pifont}
\usepackage[usenames,dvipsnames]{xcolor}
\usepackage[abs]{overpic}  
\usepackage{bbold}
\usepackage{pdflscape}   
\usepackage{mathabx} 

\usepackage{wasysym}
\usepackage{ulem}

\usepackage[all]{xy}

\usepackage{hyperref}

\usepackage{caption}

\usepackage{setspace}

\usepackage{enumitem}

\usepackage{mathrsfs}

\usepackage{rotating}
\usepackage{mathtools}

\usepackage{setspace}

\usepackage[section]{algorithm}
\usepackage{algorithmicx,algpseudocode}
\usepackage{tikz}
\usetikzlibrary{shapes, quotes}

\usepackage[all]{xy}

\def \R  {{\mathbb R}}

\def \Z  {{\mathbb Z}}

\def \C  {{\mathbb C}}

\def \CP {{{\mathbb C}{\mathbb P}}}

\def \CP {{\mathbb C}{\mathbb P}}

\def \fk  {{\mathfrak k}}

\def \G {{\mathcal G}}

\def \cU {{\mathcal U}}

\renewcommand{\mod}{/\! /}

\renewcommand{\emph}{\textbf}

\DeclareMathOperator \rank {rank}

\DeclareMathOperator \id {id}

\DeclareMathOperator \pt {pt}

\newcommand{\acts}{\mathbin{\raisebox{-.5pt}{\reflectbox{\begin{sideways}$\circlearrowleft$\end{sideways}}}}}

\numberwithin{figure}{section}
\numberwithin{table}{section}
\numberwithin{equation}{section}
\swapnumbers

\makeatletter
\let\c@equation\c@figure
\makeatother

\makeatletter
\let\c@table\c@figure
\makeatother

\makeatletter
\let\c@algorithm\c@figure
\makeatother

\newtheorem{Lemma}[equation]{Lemma}
\newtheorem{Theorem}[equation]{Theorem}
\newtheorem*{thm*}{Theorem}

\newtheorem{Question*}{Question}

\newtheorem{Proposition}[equation]{Proposition}
\newtheorem{Corollary}[equation]{Corollary}

\newtheorem*{Lemma*}{Lemma}
\newtheorem*{Corollary*}{Corollary}
\newtheorem{definition}[equation]{Definition}

\theoremstyle{definition}

\newtheorem{Definition}[equation]{Definition}

\newtheorem{Remark}[equation]{Remark}

\newcommand{\labell}[1] {\label{#1}} 

\begin{document}

\title[Equivariant cohomology distinguishes four-dimensional Hamiltonian $\mathbf{S^1}$-manifolds]{Equivariant cohomological rigidity for four-dimensional Hamiltonian  $\mathbf{S^1}$-manifolds}

\date{\today}

\author{Tara S. Holm}
\address{Department of Mathematics, Cornell University, Ithaca, NY  14853-4201, USA}
\email{tara.holm@cornell.edu}

\author{Liat Kessler}
\address{Department of Mathematics, Physics, and Computer Science, University of Haifa, Israel}
\email{lkessler@math.haifa.ac.il}

\author{Susan Tolman}
\address{Department of Mathematics,
University of Illinois,
Urbana, IL 61801, USA}
\email{tolman@illinois.edu}

\begin{abstract}
For manifolds equipped with group actions,  we have the following natural question: To what  extent  does  the equivariant cohomology determine the equivariant diffeotype? We resolve this question for Hamiltonian circle actions on compact, connected symplectic four-manifolds. They are equivariantly diffeomorphic if and only if their equivariant cohomology
rings are isomorphic as algebras over the equivariant cohomology of a point. 
In fact, we prove a stronger claim: each isomorphism between their equivariant cohomology rings is induced by an equivariant diffeomorphism.
 \end{abstract}



\maketitle

\section{Introduction}

For any algebraic invariant attached to a manifold, a key question is: To what extent  does the invariant determine the manifold? In general, cohomology does not detect the diffeotype of a manifold.  However, in recent years there has been substantial activity on establishing 
when certain classes of manifolds are {\bf cohomologically rigid}, 
that is, determined by their cohomology.
In particular, Masuda and Suh conjectured that two toric manifolds (smooth toric varieties) 
are diffeomorphic if their integral cohomology rings are isomorphic \cite{Masuda-Suh}; 
this conjecture has been affirmed in many special cases \cite{CMS}.

For manifolds with group actions, the analogous question is: To what  extent  does  the equivariant cohomology determine the equivariant diffeotype? 
For example, compact toric manifolds are {\bf equivariantly cohomologically rigid}: two compact  toric manifolds
whose equivariant cohomology rings are isomorphic as algebras are equivariantly diffeomorphic (indeed, this holds for the broader class of quasitoric manifolds) \cite[Theorem~4.1]{masuda}, \cite[Theorem~6.6]{Davis}, and \cite[Remark~2.5(1)]{Ma20}.
Moreover,  if the algebra isomorphism preserves the first equivariant Chern class, then they are equivariantly isomorphic as varieties; see \cite{masuda} and \cite[Remark~2.5(3)]{Ma20}.
Similarly, for compact, connected, four-dimensional Hamiltonian $S^1$-manifolds, 
if an equivariant cohomology algebra isomorphism
preserves the first equivariant Chern class, then
the manifolds are equivariantly biholomorphic (with respect to some compatible invariant complex structures); see \cite[Example 1.2 and Corollary 7.34]{HK}.
Note also that the diffeotype of such a manifold
is determined by the cohomology ring $H^*(M;\Z)$;
see Proposition~\ref{prop:diff}.

In order to state our results, we recall a few definitions.
A circle
action\footnote{In this paper, we will always assume that actions are effective.}
on a symplectic manifold $(M,\omega)$ is \emph{Hamiltonian} 
if there exists a \emph{moment map}, that is, a function 
 $\Phi \colon M \to \R$ satisfying
$- \iota(\xi) \omega=d \Phi $, where $\xi$ is the vector field generating the action.
In this case, we call $S^1 \acts (M,\omega)$ a 
\emph{Hamiltonian $\mathbf {S^1}$-manifold}.
Finally, recall that the (Borel)  {\bf equivariant cohomology} of an $S^1$-manifold $M$ is
a graded 
ring $H_{S^1}^*(M;\Z )$
naturally endowed with the structure of an algebra over $H^*_{S^1}(\pt;\Z)$, the equivariant
cohomology of a point.

In this paper we prove that equivariant cohomological rigidity holds for compact, connected, four-dimensional Hamiltonian $S^1$-manifolds:  two such $S^1$-manifolds
whose equivariant cohomology rings are isomorphic as algebras
are equivariantly diffeomorphic.   Indeed, our main result is that  {\bf strong equivariant cohomological rigidity} holds, that is,
the equivariant diffeomorphism can be chosen to induce the
given isomorphism in equivariant cohomology.

\begin{Theorem}
\label{thm:strong}
Let $M$ and $\widetilde{M}$ be compact, connected, four-dimensional 
Hamiltonian $S^1$-manifolds.  
Given an algebra isomorphism from
 $H_{S^1}^*(M;\Z)$ to $H_{S^1}^*(\widetilde{M};\Z)$, there exists an equivariant diffeomorphism 
 from $\widetilde{M}$ to $M$ 
 that induces the given algebra isomorphism.
\end{Theorem}

Our proof  of Theorem~\ref{thm:strong} 
relies on an earlier paper by
the first two authors \cite{HK}. 
In that paper, they define a combinatorial invariant of compact four-dimensional  Hamiltonian $S^1$-manifolds, called a {\bf dull graph}, which can be obtained from Karshon's decorated graph \cite{karshon} by 
forgetting selected information; see Definition \ref{def:dull}.
They then  establish a close correspondence between
dull graphs and  equivariant cohomology.
The other key ingredient in the proof  of Theorem~\ref{thm:strong} is the following theorem,
which  shows
that the dull graph determines the equivariant diffeotype of a four-dimensional Hamiltonian $S^1$-manifold.

\begin{Theorem} \labell{Thm:main}
Let $M$ and $\widetilde{M}$ be  compact, connected,  four-dimensional  Hamiltonian $S^1$-manifolds.
Given an isomorphism between  the dull graphs associated to  $M$ and $\widetilde{M}$,
there exists
an orientation-preserving, equivariant diffeomorphism  from $M$ to $\widetilde M$
that induces the given isomorphism.
\end{Theorem}

A symplectic manifold $(M,\omega)$ is naturally oriented by $\omega^{\frac{1}{2} \dim M}$.
An isomorphism of algebras
 $\Lambda \colon H_{S^1}^{*}(M;\Z) \to H_{S^1}^{*}(\widetilde{M};\Z)$ 
 is 
 {\bf orientation-preserving} if the induced isomorphism 
in ordinary cohomology is orientation-preserving.  
Otherwise, it is {\bf orientation-reversing}. 
Note that an equivariant diffeomorphism $f \colon \widetilde M \to M$ is orientation-preserving exactly if the induced map $f^*
\colon H_{S^1}^{*}(M;\Z) \to H_{S^1}^{*}(\widetilde{M};\Z)$  is orientation-pre\-ser\-ving.
 Theorem~\ref{Thm:main} has the following corollary.

\begin{Corollary}\label{cor:main}
Let $M$ and $\widetilde{M}$ be  compact, connected,  four-dimensional  Hamiltonian $S^1$-manifolds.
The following are equivalent:
\begin{enumerate}
\item  $M$ and $\widetilde{M}$ are equivariantly diffeomorphic; \label{it:diffeo} 

\item $M$ and $\widetilde{M}$ are equivariantly homeomorphic; \label{it:homeo}

\item  $H_{S^1}^{*}(M;\Z)$  and $H_{S^1}^{*}(\widetilde{M};\Z)$   are isomorphic as algebras;
and \label{it:eq coh}

\item the dull graphs of $M$ and $\widetilde{M}$ are isomorphic  as labeled graphs. \label{it:dull graph}
\end{enumerate}
 Moreover the diffeomorphism in \eqref{it:diffeo}, the homeomorphism in \eqref{it:homeo}, and the algebra isomorphism in \eqref{it:eq coh}
can always be chosen to be orientation-preserving.
\end{Corollary}

\begin{proof}
The implications \eqref{it:diffeo} $\Longrightarrow$ \eqref{it:homeo} and \eqref{it:homeo} $\Longrightarrow$ \eqref{it:eq coh} are immediate, while
the implication \eqref{it:eq coh} $\Longrightarrow$ \eqref{it:dull graph} is a direct consequence of \cite[Theorem~1.1]{HK}. 
The implication \eqref{it:dull graph} $\Longrightarrow$ \eqref{it:diffeo}, and the fact that the maps can be chosen to be orientation-preserving  
is by Theorem~\ref{Thm:main}.
\end{proof}

 In particular, equivariant cohomological rigidity holds for  compact, connected four-dimensional Hamiltonian $S^1$-manifolds. More care is required to show that the strong version of equivariant cohomological rigidity holds; we prove this in section~\ref{sec:thm1}.

\begin{Remark}
The maps in 
\eqref{it:diffeo}, \eqref{it:homeo}, 
and \eqref{it:eq coh} in Corollary~\ref{cor:main} can all be chosen to be orientation-reversing
if and only if $b_2:=\dim H^{2}(M)=2$.  
One direction follows from Lemma~\ref{2reversing}, where we show that if $b_2=2$, 
the equivariant diffeomorphism in \eqref{it:diffeo}
can be chosen to be orientation-reversing.  
The other direction follows from the fact that if $b_2 \neq 2$, then
there cannot be an orientation-reversing isomorphism of the (equivariant) cohomology ring; see \cite[Remark~7.27]{HK}.
\end{Remark}

\begin{Remark}
Fintushel classified locally smooth circle actions on compact, oriented four-manifolds up to orientation-preserving equivariant homeomorphism in terms of their {\bf weighted orbit spaces} \cite{Fi,Fi2}.
The weighted orbit space
consists of the orbit space $M/S^1$ 
(a three-manifold with boundary), 
labeled with certain isotropy data, and a characteristic class.
In the symplectic setting,
it is straightforward to extract the dull graph from the weighted orbit space.  Recovering Fintushel's 
invariants from the dull graph would give an alternate proof that if two manifolds have the same dull 
graph, then there is an orientation-preserving, equivariant homeomorphism between them; cf.\ Theorem~\ref{Thm:main}.  
Although this endeavor is beyond the scope of the present paper,
the quotient $M/S^1$ is 
homeomorphic to $S^3$ if the fixed points are isolated,
$D^3$ if the action has one fixed surface, and $[0,1] \times \Sigma_g$, if the action has two fixed 
surfaces of genus $g$, by  \cite{KT}.  See also \cite{Suss}.
\end{Remark}

 The remainder of the paper is organized as follows.
In Section~\ref{sec:background},  we introduce oriented dull graphs.  
We then reduce the proof of Theorem 1.2 to
showing that there is  an orientation-preserving equivariant diffeomorphism between two manifolds if their 
oriented dull graphs differ by  an operation called a ``partial flip". 
In Section~\ref{sec:partial}, we construct an equivariant diffeomorphism between any two such manifolds,
completing the proof of Theorem~\ref{Thm:main} and hence of Corollary~\ref{cor:main}.  
The construction of this equivariant diffeomorphism relies on the local model of a neighborhood 
of a chain of isotropy spheres, developed in Section~\ref{sec:corner}.  
These local models are a type of toric manifolds that we call ``corner manifolds".  
Finally, we prove Theorem \ref{thm:strong} in Section \ref{sec:thm1},
taking care to ensure
that the equivariant diffeomorphism we construct induces the given isomorphism of equivariant cohomology.

\subsection*{Acknowledgements}
The authors thank Ben Dozier, Yael Karshon, and Mikiya Masuda for illuminating discussions and
helpful suggestions.
The first author was supported in part
by the National Science Foundation under Grants DMS--1711317 and DMS--2204360. 
The second author was supported in part by the Israel Science Foundation, Grant  570/20, and
an NSF-BSF Grant 2021730. 
The third author was supported in part by Simons Foundation, Grant 637996, and
by the National Science Foundation under Grant
DMS-2204359.
Any opinions, 
findings, and conclusions or recommendations expressed in this material are 
those of the authors and do
not necessarily reflect the views of the Israel Science Foundation, the National Science Foundation,
or the Simons Foundation.

\section{Oriented dull graphs}\label{sec:background}

In this section, we first give a formal definition of the dull graph associated to a compact, connected four-dimensional Hamiltonian $S^1$-manifold, and then describe the natural orientation on this graph.    As we show, we can obtain any other orientation on the dull graph by a sequence of operations: If necessary, we take the ``opposite orientation", and then apply a finite number of ``partial flips".  Finally, we prove that if the oriented dull graphs of two compact, 
connected four-dimensional Hamiltonian $S^1$-manifolds are either isomorphic or have opposite orientations,
then there is an orientation-preserving equivariant diffeomorphism between the manifolds. 
This reduces the proof of  Theorem~\ref{Thm:main} to showing 
that there is  an orientation-preserving equivariant diffeomorphism between two manifolds if their oriented dull graphs differ by a partial flip.  
Sections~\ref{sec:corner} and \ref{sec:partial} will be devoted to this last step.

Let the circle $S^1$ act on a  compact, four-dimensional symplectic manifold $(M, \omega)$ 
with moment map $\Phi \colon M \to \R$.  
Each connected component of the fixed set $M^{S^1}$ is a symplectic submanifold, and hence is either
zero-dimensional or two-dimensional.  
Every point with stabilizer $\Z_k$ for some $k \geq 2$
is contained in an  {\bf isotropy $\Z_k$-sphere}, that is,
 an $S^1$-invariant symplectic  two-sphere with generic stabilizer $\Z_k$.
Note that each isotropy sphere contains exactly two isolated fixed points, and 
each isolated fixed point is contained in at most two isotropy spheres. 
The moment map is a Morse-Bott function and the index of each fixed component is twice 
the number of negative weights.  As Atiyah showed, this implies that the level sets of $\Phi$ are 
connected, and so $\Phi$ has  exactly one local minimum and one local maximum \cite{atiyah}.  
Hence, there are exactly two extremal fixed components and each non-extremal fixed point 
has exactly 
one positive weight and one negative weight.    
In particular, every fixed surface is extremal.

We now introduce dull graphs and orientations
on them.

\begin{definition} \label{def:dull}
The {\bf dull graph} associated to a compact, connected, four-dimensional
Hamiltonian $S^1$-manifold $(M,\omega)$ has one vertex for each fixed 
component and an edge between two vertices exactly if the vertices correspond to isolated fixed points and 
there is an isotropy $\Z_k$-sphere that 
contains those points. 
The dull graph has labels as follows: 
\begin{itemize}
\item ``Thin" vertices represent the isolated fixed points; each such vertex is labeled by whether the moment
map value is extremal or not.
\item ``Fat" vertices represent fixed surfaces; each such vertex is labeled by 
the self-intersection $e$ and the genus $g$ of the surface.
These are always extremal.

\item  Edges between vertices are labeled with integers $k\geq 2$ where the label $k$ indicates the vertices correspond to the fixed points
of an isotropy $\Z_k$-sphere.
\end{itemize}
We will denote the data of a dull graph as $\G=(V,E,L)$.
\end{definition}

Two dull graphs are {\bf isomorphic} if there is a graph isomorphism that respects the labels.

\begin{Remark} \labell{rem:or-pre}
An orientation-preserving equivariant diffeomorphism between four\--dim\-en\-sion\-al Hamiltonian
$S^1$-manifolds induces an  isomorphism  
of dull graphs.  
This  is straight-forward to check once we note
that an isolated fixed point is extremal exactly if the product of the weights is positive.
\end{Remark}

In \cite{karshon}, Karshon assigns a {\bf decorated graph} to each compact, connected, four\--dim\-en\-sion\-al
Hamiltonian $S^1$-manifolds. These graphs contain all the information in the dull graph plus two additional 
labels: the moment value of each fixed component  and the (properly normalized) area $A$  of each fixed surface.  
Hence, the dull graph can be obtained from the decorated graph by forgetting the extra information. 
Note that Karshon does not need to specify the ``extremal" or self-intersection labels because they can be 
obtained from the other data \cite[Lemma~2.18]{karshon}.  
See Figure~\ref{fig:extended and dull}.

\begin{center}
\begin{figure}[h]
\includegraphics[height=5cm]{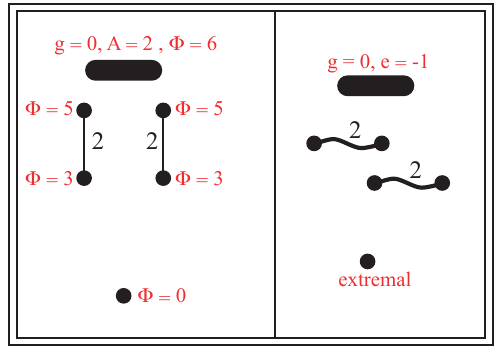} 
\caption[.]{The decorated  graph
and dull graph for a
compact, connected, four-dimensional 
 Hamiltonian $S^1$-manifold.  
The (symplectic) manifold is obtained from a  
$\CP^2$  of symplectic size $6$ by equivariant blowups of 
symplectic sizes $2$, $2$, $1$, and $1$.
}
\label{fig:extended and dull}
\end{figure}
\end{center}

\newpage

As Karshon proves in \cite[Theorem 4.1]{karshon}, decorated graphs classify compact, connected four-dimensional 
Hamiltonian $S^1$-manifolds; see also \cite{ahara hattori, audin}.  The following is an expatiation of her result.

\begin{Theorem} [Karshon] \label{thm:elucid}
Let $(M,\omega,\Phi)$ and $(\widetilde{M},\widetilde{\omega},\widetilde{\Phi})$ be compact, connected four\--dim\-en\-sion\-al Hamiltonian $S^1$-manifolds.
Let $f \colon M_{\min} \to {\widetilde{M}}_{\min}$ be an orientation-preserving diffeomorphism
between the minimal fixed components of $M$ and $\widetilde{M}$,
respectively.
Given an isomorphism between the decorated graphs associated to $M$ 
and $\widetilde{M}$,  there exists an equivariant symplectomorphism $h \colon M \to \widetilde{M}$ inducing  that isomorphism
such that $h|_{M_{\min}}$
is isotopic to $f$.
\end{Theorem}

\begin{proof}
 By Moser's method, the orientation-preserving diffeomorphism $f \colon M_{\min} \to {\widetilde{M}}_{\min}$ is isotopic to a symplectomorphism.
In \cite[proof of Proposition 4.3]{karshon}, Karshon prescribes how to extend such a symplectomorphism  to an equivariant orientation-preserving diffeomorphism 
 $H \colon M \to \widetilde{M}$
 that intertwines the moment maps and induces the given isomorphism on the graphs.
The forms $\omega_t=(1-t)\omega+t \widetilde{\omega}$, for $0\leq t \leq 1$ are all in the same cohomology class, by \cite[Lemma 4.12]{karshon}.
Moreover, $\omega_t$ is nondegenerate for all $t$ \cite[proof of Proposition 4.11]{karshon}. Hence, by Moser's method, the equivariant diffeomorphism $H$ is equivariantly isotopic to an equivariant symplectomorphism  $h \colon M \to \widetilde{M}$.
\end{proof}

Oriented dull graphs are an intermediate step between dull graphs and decorated graphs.

 \begin{definition}
An  {\bf orientation} $\mathcal O$ on a dull graph $\G$ consists of
\begin{itemize}
\item Adding the label  ``mininum" to one of the extremal vertices and
 the label ``maximum" to the other; and 
\item Giving $\G$ the structure of a directed graph such that 
the minimum is not the head of any edge, the maximum is not the tail of any edge, 
and each non-extremal vertex is the head of at most one edge, and the tail of at most one edge.
\end{itemize}
\end{definition}

Given a Hamiltonian $S^1$- manifold $(M,\omega,\Phi)$,
the moment map $\Phi \colon M \to \R$ induces a natural orientation on the associated dull graph $\G$.  We
label the extremal vertices according to whether their moment image is maximal or minimal, and direct each edge so that 
the moment image of the head is greater than the image of the tail.  If a non-extremal isolated fixed point $p$ is 
contained in two isotropy spheres, then the moment image of one sphere lies above $\Phi(p)$ and the 
moment image of the other sphere lies below. 
Hence, this gives an orientation $\mathcal O$ on $\G$ in the sense of the definition above.

As Holm and Kessler prove, one consequence of Theorem~\ref{thm:elucid} is that oriented dull graphs distinguish 
four-dimensional Hamiltonian $S^1$-manifolds up to equivariant diffeomorphism.

\begin{Lemma}{\rm(\cite[Remark~3.5 and Proposition~3.9]{HK}\rm)}\labell{lem:hef} 
Given an orientation-preserving isomorphism between 
the  dull graphs associated to two
compact, connected four-dimensional Hamiltonian $S^1$-manifolds,
there is an orientation-preserving, equivariant diffeomorphism between the manifolds that induces the given isomorphism.
\end{Lemma}

We say that a connected component $C$ of the set of points with non-trivial stabilizer is  a {\bf free chain}  if the component does not contain either the minimal or maximal fixed component.
A {\bf free point}  is an isolated fixed point that has isotropy weights $1$ and $-1$.
Since a fixed point $p$ is the maximum (resp.\ minimum) of an isotropy $\Z_k$ sphere exactly if one of the isotropy weights at $p$ is $-k$ (resp.\ +$k$), 
every free point is a trivial free chain.
Similarly, let  $\G=(V,E,L)$ be the dull graph associated to the   Hamiltonian $S^1$-action on $M$. By the discussion above,  the dull graph $\G$ has two ``extremal" vertices and each vertex has degree at most two.
A {\bf free chain} of the dull graph $\G$ is a connected component $C$ of $\G$ that does not contain either extremal vertex.
A {\bf free point} is a non-extremal point with no adjacent edges.

Given one  orientation on a dull graph $\G$, there is another orientation on $\G$ given by exchanging the ``minimum" and ``maximum" labels and reversing the direction on  every edge of the graph.  We call this the {\bf opposite orientation}.
Additionally, given a free chain $C$  that is not a free point, there is another orientation obtained by reversing the direction on each  edge in the chain $C$ and leaving all other aspects of the orientation the same. We call this the  {\bf partial flip along $\mathbf C$}.  These operations on orientations are shown in Figure~\ref{fig:orientations}.

\begin{center}
\begin{figure}[h]
\includegraphics[width=15cm]{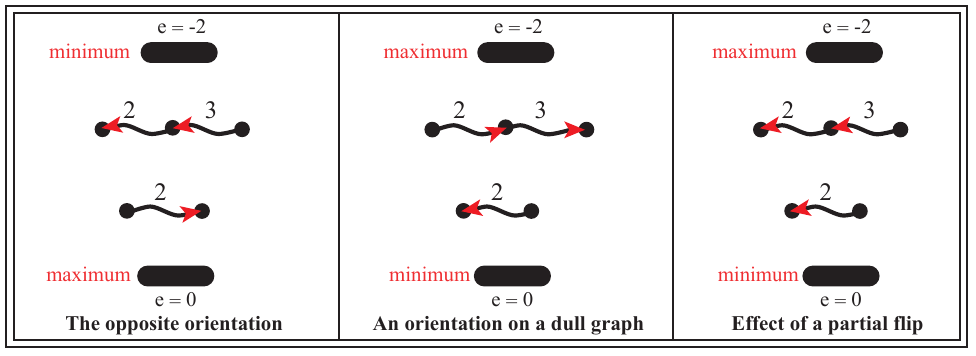} 
\caption[.]{An orientation $\mathcal{O}$ on a dull graph $\G$ is shown in red on the middle
figure.
On the left, we see the opposite orientation; on the right, we see the partial flip on the two-edge chain.
}
\label{fig:orientations}
 \end{figure}
\end{center}

It is straight forward to see that any two orientations on $\G$ with the same minimum and maximal labels differ by some number of partial flips.
We record this in the following lemma.

\newpage

\begin{Lemma}
Any two orientations on a dull graph differ by
\begin{itemize}
\item possibly taking the opposite orientation; and then
\item applying finitely many partial flips.
\end{itemize}
\end{Lemma}

We aim to show that  two Hamiltonian $S^1$-manifolds are equivariantly diffeomorphic if their oriented dull graphs differ by either of these modifications.  
We first observe that this is true 
when the dull graphs have opposite orientations.

\begin{Lemma} \label{lem:fullflip}
Given a compact, connected four-dimensional Hamiltonian $S^1$-manifold $(M, \omega, \Phi)$, there exists
another compact, connected four-dimensional Hamiltonian $S^1$-manifold $(\widetilde M, \widetilde \omega, \widetilde \Phi)$ with isomorphic dull graph
and an orientation-preserving, equivariant diffeomorphism from $M$ to $\widetilde M$ that induces an orientation-reversing isomorphism on the associated dull graphs.
\end{Lemma}

\begin{proof} Let $(\widetilde M, \widetilde \omega, \widetilde \Phi)$ be $(M, - \omega, - \Phi)$ and take the identity map.
\end{proof}

\noindent In the next two sections, we will show that we have an equivariant diffeomorphism between four-dimensional Hamiltonian $S^1$-manifolds if their oriented dull graphs differ by a partial flip along a free chain.

 It is well-known that the diffeomorphism type of a compact, 
 connected four-dimensional
Hamiltonian $S^1$-manifold $M$ is determined by 
the cohomology ring $H^*(M;\Z)$.  For the reader's convenience, 
we conclude this section with an explicit proof,
making use of Lemma~\ref{lem:hef}.

\begin{Proposition} \label{prop:diff}
Let $M$ and $\widetilde M$ be compact, connected, four-dimensional Hamiltonian $S^1$-manifolds. Then $M$ and $\widetilde M$ are diffeomorphic if and only if the cohomology rings $H^*(M;\Z)$ and $H^*(\widetilde M;\Z)$ are isomorphic.
In this case, we can choose the diffeomorphism to be orientation-preserving.
\end{Proposition}

\begin{proof}
By \cite[Theorem~6.3  and Lemma~6.15]{karshon},
every compact, connected, four-dimensional Hamiltonian $S^1$-manifold can either be obtained as a blowup of $\CP^2$ or a blowup of an $S^2$ bundle over a compact, oriented surface $\Sigma$. 
We will proceed by analyzing the possibilities depending on the rank of $H^2(M;\Z)$.
If $H^2(M;\Z) = \Z$, then $M$ is diffeomorphic to $\CP^2$.  

Proceeding, we may assume that $\rank H^2(M;\Z) > 1$.  Since the blow up of $\CP^2$ is 
an $S^2$-bundle over a sphere, this implies that $M$ can be obtained as a blowup of 
an $S^2$-bundle over a compact, oriented surface $\Sigma$. The genus $g$ of the surface 
is $\frac{1}{2} \rank (H^1(M;\Z))$.  
If $\rank(H^2(M;\Z)) = 2$, then $M$ is diffeomorphic to an $S^2$-bundle over $\Sigma$. 
There are exactly two oriented $S^2$-bundles over a compact, oriented surface $\Sigma$. 
If $H^2(M;\Z)$ with the intersection form is an even lattice, then $M$ must be the 
trivial bundle $S^2 \times \Sigma$, and if $H^2(M;\Z)$ is an odd lattice, then $M$ must 
be the non-trivial bundle.

Finally, the one-fold blowups of the trivial and non-trivial oriented $S^2$ bundles 
over $\Sigma$ are diffeomorphic by an orientation-preserving diffeomorphism.  
This can be seen, for example, by considering a circle action with fixed set $\Sigma \sqcup \Sigma$, blowing up at a fixed point, and applying  Lemma~\ref{lem:hef}.  Therefore, if $\rank(H^2(M;\Z)) \geq 3$, then $M$ is a blowup of $S^2 \times \Sigma$ at $(\rank(H^2(M;\Z)) - 2)$ points.  

We have thus shown that $\rank(H^1(M;\Z))$, $\rank(H^2(M;\Z))$, and the intersection form on 
$H^2(M;\Z)$ together determine $M$ up to orientation-preserving diffeomorphism.  The other implication is immediate: diffeomorphic
manifolds have isomorphic cohomology rings.
\end{proof}

\section{Corner manifolds}\labell{sec:corner}

 In this section we use the fact that a free chain of isotropy spheres 
can be  equivariantly blown down to a free point
to identify 
 neighborhoods of  chains with smooth symplectic toric manifolds. We then explore the symplectic analogue of holomorphic blowup maps on such toric manifolds.
 We begin by introducing  the symplectic toric manifolds that arise, 
which we call corner manifolds.

\begin{Definition}
A \emph{corner manifold} $N$ is a connected four-dimensional symplectic manifold $N$ with a proper moment map $\Phi \colon N \to \R^2$ such that the
moment image $\Delta := \Phi(N)$ is contained in the positive quadrant $\left(\R_{\geq 0} \right)^2$ and contains all $(w_1,w_2) \in \left(\R_{\geq 0}\right)^2 $ such that $w_1+w_2$ is sufficiently large.
\end{Definition}

In particular,  $\C^2$ with the standard action and symplectic form $(S^1)^2 \acts (\C^2, \omega_0)$, and with the homogeneous moment map,
is itself a corner manifold.

\vskip 0.1in
 \begin{center}
\begin{figure}[h]
\begin{overpic}[
scale=1,unit=1mm]{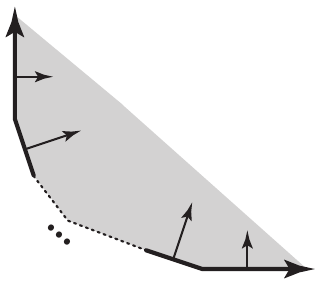} 
   \put(10,35){{\small{$v_1$}}}
   \put(15,25){{\small{$v_2$}}}
   \put(32,15){{\small{$v_{{n-1}}$}}}
   \put(41,11){{\small{$v_n$}}}
\end{overpic} 
\caption[.]{The  polyhedral set $\Delta$ with the labeled, inward-pointing normal vectors indicated. 
}
\label{fig:corner}
 \end{figure}
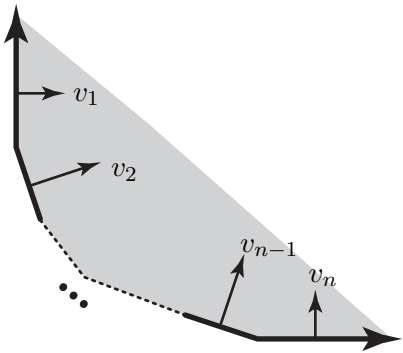
 \end{center}

\vskip 0.1in

 The moment image of a corner manifold is a Delzant convex polyhedral set
$$\Delta = \bigcap_{i=1}^n \{ w \in \R^2 \mid \langle w, v_i \rangle \geq a_i \},$$ 
where $v_1,\ldots,v_{n} \in \Z^2$ are the primitive inward-pointing  normal vectors to the edges of 
 $\Delta$  and $a_1,\dots,a_n \in \R$.  We will always order the $v_i$'s counter-clockwise and set  $v_1=(1,0)$, $v_{n}=(0,1)$, and  $a_1=0=a_n$.  We will assume that each $v_i$ defines a non-trivial edge of $\Delta$: that is, the set $\{ w \in \R^2 \mid \langle w, v_i \rangle = a_i \}$ is a closed, one-dimensional subset of $\R^2$. 
 Here, a rational compact polyhedral set in $\R^2$ is \emph{Delzant} if the primitive inward-pointing normal vectors to each pair of adjacent edges form a basis of $\Z^2$.
For each $i$ such that $1 < i < n$, the
set 
\begin{equation}\label{eq:Si1}
S_i := \{p \in N \,|\, \langle \Phi(p), v_i \rangle = a_i \}
\end{equation}
is
a symplectic two-sphere that is fixed by the circle in $(S^1)^2$ generated by the vector $v_i\in \R^2 = \mathrm{Lie}((S^1)^2)$.  
The  {\bf chain of isotropy spheres}  of $N$  is the union $D$ of these $(n-2)$ two-spheres.
Since $v_i \in (\Z_{>0})^2$ for all $1 < i < n$, the chain $D$ is the set of points with non-trivial stabilizer under the anti-diagonal circle action.

\newpage

Conversely, consider a Delzant convex polyhedral set
$$\Delta := \bigcap_{i=1}^n\  \{ w \in \R^2 \mid \langle w, v_i \rangle \geq a_i \} \  ,$$
where $v_1,\ldots,v_{n} \in \Z^2$ are  the primitive inward-pointing  normal vectors to $\Delta$, ordered
counter-clockwise and $a_i \in \R$ for all $i$. Moreover, assume that $v_1=(1,0)$, $v_{n}=(0,1)$, and  $a_1=0=a_n$. 
Set $v_i=(\alpha_i,\beta_i)$.
Note that $\Delta$ is contained in the positive quadrant $\left(\R_{\geq 0} \right)^2$ and contains all $(w_1,w_2) \in \left(\R_{\geq 0}\right)^2 $ such that $w_1+ w_2$ is sufficiently large.
Following Delzant \cite{De},   we construct a
corner manifold $(M_\Delta,\omega_\Delta,\Phi_\Delta)$,
with moment image $\Delta$,
by symplectic reduction.  
Consider the standard $(S^1)^n$-action on $(\C^n, \omega_0)$ with moment map $\phi \colon \C^n \to \R^n$ sending $(x_1,\dots,x_n)$ to $\frac{1}{2}(|x_1|^2,\dots,|x_n|^2)$.
Let $K=K_{\Delta}$ be the kernel of the map $(S^1)^{n} \to (S^1)^2$ induced by the map
$\pi \colon \R^{n} \to \R^2$ defined by $\pi(e_i)=v_i$,
where $(e_1,\ldots,e_{n})$ is the standard basis of $\R^{n}$.
(We identify $S^1$ with ${{\R} /{\Z}}$.) 
We have dual short exact sequences
\begin{equation}\label{eq:lie-alg}
\begin{array}{c}
\xymatrix{
0\ar[r] &   \fk  \ar[r]^(.45){\iota}  & \R^n  \ar[r]^(.45){\pi}
& \R^2  \ar[r] & 0 &  \mbox{and} \\
0 &   \fk^* \ar[l]   & (\R^n)^* \ar[l]_{\iota^*}  
& (\R^2)^* \ar[l]_(.45){\pi^*} & 0, \ar[l] &
}
\end{array}
\end{equation}
where $\fk$ is the Lie algebra of $K$ and $\iota \colon \fk\to \R^n$ 
is the inclusion map. Note that $\iota(\fk)$ has 
basis $\alpha_ie_1 - e_i + \beta_ie_n$ for $i=2,\dots,n-1$.
Let $M_\Delta$ be symplectic reduction of $\C^n$ by $K$ at $\iota^*(-a_1,\dots,-a_n)$. 
Explicitly,
we have the moment map
$\mu:\C^n\to\fk^*$, which is the composition of $\phi$ and $\iota^*$.
The level set $\mu^{-1}(\iota^*(-a_1,\dots,-a_{n}))$ is
\begin{equation}\label{eq:level}
U_\Delta
= \left\{ x\in \C^n\ \left| \ 
\textstyle{\frac{1}{2}}\left( \alpha_i |x_1|^2 - |x_i|^2+ \beta_i |x_n|^2\right)=a_i \mbox{ for all } 1 < i < n
\right.
\right\},
\end{equation}
and so
\begin{equation}\label{eq:red}
M_\Delta=\mu^{-1}(\iota^*(-a_1,\dots,-a_{n}))/K = U_\Delta/K.
\end{equation}
By \eqref{eq:lie-alg},
we may identify $(S^1)^n/K$ with $(S^1)^2$ and hence interpret the residual Hamiltonian $(S^1)^n/K$-action on $M_\Delta$ as an $(S^1)^2$-action. 
Under this identification, $(S^1)^2$ acts on $M_\Delta$ by
\begin{equation}\label{eq:action}
(\lambda_1,\lambda_2) \cdot [x_1,\dots,x_n] = [\lambda_1 x_1, x_2, \dots, x_{n-1}, \lambda_2 x_n]
\end{equation}
with moment map
\begin{equation}\label{eq:mmap}
[x_1,\dots,x_n] \mapsto \textstyle \frac{1}{2}(|x_1|^2,|x_n|^2),
\end{equation}
and  isotropy spheres
\begin{equation*}
S_i = \left\{ [x] \in M_\Delta \ \Big|\ x_i = 0 \right\} \ \forall \ 1 \leq i \leq n.
\end{equation*}

Moreover, up to $(S^1)^2$-equivariant symplectomorphism,
the manifold $M_\Delta$ is the unique corner 
manifold with moment
image $\Delta$  \cite[Theorem 1.3]{KL:noncpt toric}.
Hence, we may identify any corner manifold with moment polytope $\Delta$ with $M_\Delta$.
As we see in the proof of Lemma~\ref{lem:strip}, 
every corner manifold 
can be obtained from the standard action
$(S^1)^2 \acts (\C^2,\omega_0)$ 
with the homogeneous moment map by a finite sequence 
of $(S^1)^2$-equivariant symplectic blowups.
Alternatively, as a complex manifold, 
it can be obtained by a finite sequence
of holomorphic blowups. 
From this perspective, 
the chain of isotropy spheres $D$  is the preimage 
of the origin $0\in \C^2$ under the composition of holomorphic blowup maps.
Given a corner manifold $(S^1)^2\acts N$, we will often restrict our attention to 
actions  by 
the {\bf anti-diagonal circle} $\{  (\lambda,\lambda^{-1}) :\lambda\in S^1 \}\subset (S^1)^2$ 
and  the
{\bf diagonal circle}   $ \{  (\lambda,\lambda) : \lambda\in S^1 \}\subset (S^1)^2$.

\begin{definition}\labell{def:triangular}
A \emph{triangular neighborhood (of size $\bm{\varepsilon}$)} in $N$  is a neighborhood $U_{\varepsilon}$ of $D$ that is the moment 
preimage in a corner manifold $N$ of the triangle
\begin{equation} \labell{eq:triangle}
\{ (w_1,w_2)\,|\, w_1 \geq 0, \, w_2 \geq 0, \, w_1+w_2 < \varepsilon\}.
\end{equation}
For example, triangular neighborhoods in $\C^2$ are balls. More generally, $U_\varepsilon$  is the preimage of $[0,\varepsilon)$ 
under the moment map for  the diagonal circle action.  
\end{definition}

\begin{Remark}\labell{rem:shrink}
\noindent  Since the chain of isotropy spheres $D$  in a corner manifold $N$ is 
a compact subset and the moment map is continuous, there always exists an 
$\varepsilon >0$ such that 
$U_{\varepsilon}$ is a triangular neighborhood.
Moreover, if $D \subset U_{\varepsilon} \subset N$, for $\varepsilon >0$, then there exists  
$0<\delta<\varepsilon$ such that $D \subset U_{\delta} \subset N$.
\end{Remark}

Next, we show that corner manifolds serve as local models for neighborhoods of free chains in four-dimensional Hamiltonian $S^1$-manifolds. For that, we need to introduce the notion of a chain blowup. Consider a four-dimensional symplectic manifold $(M,\omega)$ with a Hamiltonian $S^1$-action.
Let $p \in M$ be a non-extremal fixed point with weights $(m_2,-m_1)$, where $m_1, m_2 > 0$. 
For sufficiently small $\varepsilon > 0$, we can perform an {\bf equivariant blowup} of size $\varepsilon$ at $p$. 
Such a blowup is possible if and only if there is an equivariant symplectic embedding 
of an open ball $B_r$ about the origin in $\C^2$ of radius $r$, where $\frac{r^2}{2} > \varepsilon$, and
 $S^1$ acts on $B_r$ by 
$\lambda \cdot (z_1,z_2) = ( \lambda^{m_1} z_1, \lambda^{m_2} z_2)$.
In the blowup, the fixed point $p$ is replaced by two isolated fixed points connected by an isotropy sphere of area $\varepsilon$ with stabilizer $m_1 + m_2$, as indicated in the figure below.

\begin{center}
\begin{figure}[h]
\includegraphics[width=4cm]{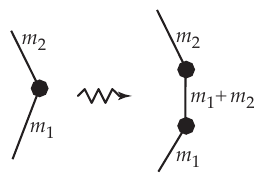} 
\caption[.]{The effect of a  blowup at a non-extremal fixed point.}
\label{fig:blowup}
 \end{figure}
\end{center}

\noindent Note that the isotropy sphere created in the equivariant blowup procedure is an exceptional sphere: it has self-intersection $-1$.  
The blowup procedure
can be reversed: any isotropy sphere of self-intersection $-1$ can be equivariantly {\bf blown down} and replaced with an isolated 
fixed point.

A \emph{chain blowup} is a  sequence $M^0,\ldots,M^k$ of Hamiltonian $S^1$-manifolds together with a 
 free chain $C^j$ in $M^j$, for each $j$, satisfying
\begin{itemize}
\item $M^{j+1}$ is obtained from $M^{j}$ by an $S^1$-equivariant symplectic blowup 
at a fixed point $v^{j}$ in $C^{j}$; and
\item $C^{j+1}$ is the free chain that contains the newly created exceptional sphere.
\end{itemize}
By a slight abuse of terminology, when such a sequence of blowups exists, we call $(M^k, C^k)$ a {\bf chain blowup} of $(M^0, C^0)$.
As we show next, every free chain comes from  a chain blowup of a free point.

\begin{Lemma}\labell{lem:ex}
Let $M$ be a compact four-dimensional Hamiltonian $S^1$-manifold and $C$ a free chain in $M$.
Then there exists a  Hamiltonian $S^1$-manifold $\widecheck{M}$ with a free point $\widecheck{C}$ so that
$(M,C)$ is a chain blowup of $(\widecheck{M},\widecheck{C})$.
\end{Lemma}

\begin{proof} 
It is enough to show that every non-trivial free chain of isotropy spheres contains an 
isotropy sphere with self-intersection $-1$. Since such a sphere can be blown down, the theorem will then  follow by
induction on the number of fixed points in the chain $C$.

If $C$ is not a free point, label the isotropy spheres in the chain $C$ as $e_1, \dots, e_n$, so that the maximal fixed point of $e_i$ is the minimal fixed point of $e_{i+1}$ for all $i$.  Let $m_i$ be the isotropy weight of $e_i$ for all $1 \leq i \leq n$. Since the negative weight at the minimal fixed point in $C$ is $-1$, we  set $m_0 = 1$; similarly, we set $m_{n+1} = 1.$
Since the action is effective, $m_i$ and $m_{i+1}$ are relatively prime for all $i$.
Hence, if $m_j$ is the largest weight in the chain,
then $m_j > m_{j-1}$ and $m_j > m_{j+1}.$
Since the $m_j$'s are positive, this implies that
$$ 0 > - \frac{m_{j-1} + m_{j+1}}{m_j} > - 2.$$
By \cite[Proof of Lemma~5.2]{karshon},
the middle term in this string of inequalities
is the self intersection of $e_j$,
which must be an integer. Therefore, the sphere has self-intersection equal to $-1$. 
\end{proof}

After possibly replacing our four-manifold by another with isomorphic oriented dull graph,  every free chain admits an invariant neighborhood that is isomorphic to a triangular neighborhood in a corner manifold, equipped with the anti-diagonal circle action.

\begin{Lemma} \labell{lem:new}
Let $(M,\omega,\Phi)$ be a compact, connected four-dimensional Hamiltonian $S^1$-manifold and $C$ a free chain in $M$.
There exists a compact, connected four-dimensional Hamiltonian $S^1$-manifold $\widetilde{M}$ whose
oriented dull graph is isomorphic to that of $M$,
and an $S^1$-equivariant symplectomorphism 
from a triangular neighborhood
in a corner manifold, equipped 
with the anti-diagonal circle action,
to a neighborhood of the free chain in
$\widetilde{M}$ corresponding to $C$.
\end{Lemma}

\begin{proof} 
By Lemma~\ref{lem:ex}, there exists a chain blowup $M^0,\dots,M^k = M$ of Hamiltonian $S^1$-manifolds together with a free chain $C^j$ in $M^j$ for each $j$, satisfying
\begin{itemize}
\item $M^{j+1}$ is obtained from $M^{j}$ by an $S^1$-equivariant symplectic blowup 
at a fixed point $v^{j}$ in $C^{j}$; 
\item $C^{j+1}$ is the free chain that contains the newly created exceptional sphere
that replaces the isolated fixed point $v^{j}$; and
\item   $C^0$ is a free point that we will denote $p$.
\end{itemize}

In particular, the isotropy weights at $p$ are $(1,-1)$ and we may assume without loss of generality that $\Phi(p)=0$.
By the Local Normal Form Theorem, 
there are complex coordinates $z_1,\, z_2$ on an invariant neighborhood  $U$ of $C^0$ in $M^0$ such that
\begin{itemize}
\item the circle action is $\lambda \cdot (z_1,z_2)=(\lambda z_1,\lambda^{-1} z_2)$,

\vspace{.01in}
\item the symplectic form is $\omega=\frac{i}{2}(dz_1 \wedge d\, \overline{z}_1+d z_2 \wedge d\, \overline{z}_2)$,

\vspace{.02in}
\item the moment map is $\Phi(z_1,z_2)=\frac{1}{2}|z_1|^2-\frac{1}{2}|z_2|^2$, and

\vspace{.03in}
\item the neighborhood $U=\left\{ (z_1,z_2)\, \big|\,  \, \frac{1}{2}|z_1|^2+\frac{1}{2}|z_2|^2 < \varepsilon\right\}$  .
\end{itemize}
The standard $(S^1)^2$-action on $U$
extends the given $S^1$-action as the anti-diagonal circle action.

Next we construct a chain blowup $U = U^0, U^1, \dots, U^k = \widetilde{U}$ with the same ordering as the given chain blowup.  
Concretely, for each $j$, let $r_j$ be the number of vertices in the chain $C^j$ whose moment map value is less than or equal to that of $v_j$.  
At each step, let $v^j \in U^j$ be the fixed point  such that there are exactly $r_j$ fixed points whose whose moment  image for the anti-diagonal circle action is less than or equal to that of $v^j$.
We perform an $(S^1)^2$-equivariant symplectic blowup at the 
fixed point $v^j$ of symplectic size small enough so that that the ball used to perform the blowup is completely contained inside of $U^j$.
Finally, to construct $\widetilde{M}$, we replace the neighborhood
$U$ in $M$ with the neighborhood  $\widetilde{U}$ equipped with
the anti-diagonal circle action.
\end{proof}

 To achieve a partial
flip of the dull graph, we  use a procedure similar  to the one in the proof of Lemma~\ref{lem:new}, blowing down a
chain and then blowing it back up, but in this case
we blow it back up ``upside down."

\begin{Lemma}\labell{prop:existence}
Given any orientation $\mathcal O$ on a dull graph $\G$, there exists a compact, connected four-dimensional
Hamiltonian $S^1$-manifold whose oriented dull graph is isomorphic to $(\G, \mathcal O)$.
\end{Lemma}

\begin{proof}
By definition, there's  a compact, connected four-dimensional Hamiltonian $S^1$-manifold $(M,\omega,\Phi)$
with dull graph $\G$.  
By replacing $\Phi$ by $-\Phi$ if necessary,
we may assume that the  orientation on $\G$
induced by the moment map $\Phi$ has
the same 
 ``maximum" and ``minimum" labels 
as the orientation $\mathcal O$,  that is,
the two orientations differ by a finite number of partial flips.

 Therefore, to construct a manifold $\widetilde M$ whose oriented dull graph
is isomorphic to $(\G, \mathcal O)$,  it suffices to show that given any free chain $C$ the following procedure produces a manifold $M'$ so that the oriented dull graphs of $M$ and $M'$ are isomorphic except that the orientations differ by a partial flip along $C$:
By Lemma~\ref{lem:ex}, there exists a  Hamiltonian $S^1$ manifold $\widecheck{M}$ with a free point $\widecheck{C}$ so that
$(M, C)$ is a chain blowup of $(\widecheck{M}, \widecheck C )$, i.e., there exists a sequence $\widecheck{M} = M^0,\dots,M^k = M$ of Hamiltonian $S^1$-manifolds 
together with a free chain $C^j$ in $M^j$ for each $j$, satisfying
\begin{itemize}
\item $M^{j+1}$ is obtained from $M^{j}$ by an $S^1$-equivariant symplectic blowup 
at a fixed point $v^{j}$ in $C^{j}$;  and
\item $C^{j+1}$ is the free chain that contains the newly created exceptional sphere.
\end{itemize}
To construct $M'$, we perform an {\bf inverted chain blowup}: a  chain blowup $$\widecheck M = (M')^0, (M')^1, \dots, (M')^k=M'$$ so that the oriented dull graphs of $(M')^j$ and $M^j$ differ by a partial flip for each $j$.
Concretely, 
let $r_j$ be the number of vertices in the chain $C^j$ whose moment map value is  less than or equal to that of $v^j$. 
Let $(M')^0 = \widecheck M$, $(C')^0=C^0$, and for each $j > 0$
let $M^{j+1}$ be the blow up at the fixed point $(v')^j \in (C')^j$ such that
 the chain $(C')^j$ has exactly $r_j$ points whose moment map value is  greater 
than or equal to that of $(v')^j$.
\end{proof}

\begin{center}
\begin{figure}[h]
\includegraphics[width=10cm]{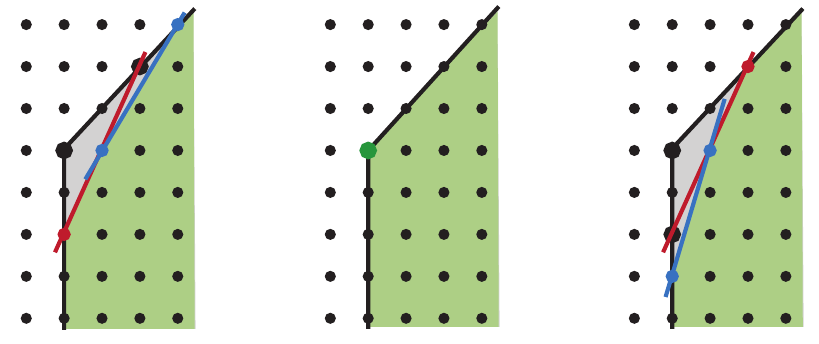}
\caption[.]{
The center figure represents a neighborhood of a free point; the $S^1$  moment map is the projection onto the vertical axis.
The figure on the left represents a chain blowup at this point. The figure on the right is the inverted chain blowup.
In
 each case, the red  and blue edges represent isotropy two-spheres with  stabilizers $\Z_2$ and $\Z_3$, respectively.
}
\label{fig:flip-chain}
 \end{figure}
\end{center}

We now explore the symplectic analogue of holomorphic blowup maps on corner manifolds. 
For the blowup of a single point, the holomorphic blowup map is a surjective map that collapses
the exceptional sphere to a point.  
By contrast, the symplectic blowup map is a symplectomorphism, defined only on the complement of the exceptional sphere.
In the following lemma, we iterate this process to
explicitly define a symplectomorphism $\psi$ from the complement of the chain of isotropy spheres in a corner manifold
onto an open subset of $\C^2$.
We call the map $\psi$ the \emph{symplectic blowup map}.

\begin{Lemma}\labell{lem:strip}
Let $N$ be a corner manifold with moment image $\Phi(N)=\Delta$; we identify $N$ with the reduced space $M_\Delta := \C^n\mod K$, equipped with the $(S^1)^n/K\cong (S^1)^2$ action.
Let $v_1,\dots,v_n$  be the inward-pointing normal
vectors to the  facets of $\Delta$, with the standard counter-clockwise order, and let $v_i = (\alpha_i, \beta_i)$ for all $i$.
Let $D$ be the chain of isotropy spheres in $N$. 
Then the map 
\begin{gather*} 
\psi \colon  N \smallsetminus D  \to    \bigcap_{i=2}^{n-1} \left\{ z \in \C^2\  \Big|\ \textstyle{\frac12}  \left\langle \big(|z_1|^2,|z_2|^2\big), v_i\ \right\rangle  > a_i \right\} \\
 [x_1,\dots, x_n]  \mapsto  \left( \frac{ x_1 x_2^{\alpha_2} \cdots x_{n-1}^{\alpha_{n-1}} }{ | x_2^{\alpha_2} \dots x_{n-1}^{\alpha_{n-1}}|},
\frac{x_{2}^{\beta_{2}} \cdots x_{n-1}^{\beta_{n-1}} x_{n}^{}}{|x_{2}^{\beta_{2}} \cdots x_{n-1}^{\beta_{n-1}}|} \right) 
\end{gather*}
is a well-defined $(S^1)^2$-equivariant symplectomorphism  
that  intertwines the moment maps.
Here,  $\C^2$ is equipped with the standard symplectic form, $(S^1)^2$-action, and the  homogeneous moment map.
\end{Lemma}

\begin{proof}
We will prove the claim by showing that $\psi$ can be written as the composition of equivariant symplectomorphisms
that intertwine the moment maps. 
If $n = 2$, the map $\psi$ is the identity map and the claim is trivial.
If $n \geq 3$, we claim that there is at least one index $1 < j < n$ so that the set of points $p \in N$ with $\langle \Phi(p), v_j \rangle = a_j$ is a two-sphere in $N$ with self intersection $-1$, or equivalently such that $v_j=v_{j-1}+v_{j+1}$.
To see this, apply  the argument in the proof of Lemma~\ref{lem:ex} to the anti-diagonal circle action on $N$.
Fix one such $j$ and
define
$$
\widecheck{\Delta} = \bigcap_{\stackrel{i=1}{i\neq j}}^n \left\{ w \in \R^2\ \Big|\ \langle w, v_i \rangle \geq a_i \right\}.
$$ 
The fact that  $v_j=v_{j-1}+v_{j+1}$
  guarantees that $\widecheck{\Delta}$ is again a Delzant polyhedral set.
Let $\widecheck{N} \simeq M_{\widecheck{\Delta}}$ be the associated corner manifold.

Before completing the proof, to
motivate the definition of $\psi$, 
we recall the holomorphic blowup map.
Consider the holomorphic map
\begin{equation} \labell{eq:singlecbup}
 [x_1,\ldots,x_{n}] \mapsto [x_1,\ldots,x_{j-2},{x_{j-1} x_{j}},{x_{j}x_{j+1}},x_{j+2},\ldots,x_{n}]
\end{equation}
from $N$ to  $\widecheck{N}$.  The space  $N$, together with the map \eqref{eq:singlecbup}, is the complex blowup 
of $\widecheck{N}$ at the unique point $[y]\in \widecheck{N}$ that has $y_{j-1}=y_j=0$. 
 The map \eqref{eq:singlecbup} sends
the exceptional sphere  $\{ [x] \in N \mid x_{j} = 0 \}$ to  $[y]$ and is a biholomorphism on the complement of these sets.
Applying this blowdown procedure inductively and composing the blowup maps, 
we get a holomorphic blowup map from $N$ to $\C^2$
\begin{equation*}\labell{eq:cbup}
[x_1,\ldots,x_{n}] \mapsto \Big({x_1^{} x_2^{\alpha_2} \cdots x_{n-1}^{\alpha_{n-1}}},
{x_{2}^{\beta_{2}}  \cdots  x_{n-1}^{\beta_{n-1}} x_{n}^{}}\Big),
\end{equation*}
where $v_i = (\alpha_i, \beta_i)$;
moreover, the chain $D \subset N$ is the preimage of the origin in $\C^2$.

We now return to
completing the proof by analyzing the analogous symplectic blowup.
We claim that the map
\begin{equation}\label{eq:Teq-symplectom}
\varpi \colon \left\{ [x] \in N \ \Big|\  x_{j}\neq 0 \right\} \to \left\{[y] \in \widecheck{N}\ \Big|\ 
\textstyle{\frac{1}{2}}\alpha_j |y_1|^2 + \textstyle{\frac{1}{2}} \beta_j|y_{n-1}|^2 > a_j \right\}
\end{equation}
given by
\begin{equation*}\labell{eq:singlesbup}\textstyle{
 \left[x_1,\ldots,x_{n+1}\right] \mapsto  \left[x_1,\ldots,x_{j-2}, \frac{x_{j-1} x_{j}}{| x_{j}|},\frac{x_{j}x_{j+1}}{|x_{j}|},x_{j+2},\ldots,x_{n}\right]
}
\end{equation*}
is a well-defined $(S^1)^2$-equivariant symplectomorphism  
that intertwines the moment maps.
To see this, consider the map
\begin{eqnarray*}
\widetilde \varpi \colon\left \{x \in \mathbb C^{n} \ \Big|\  x_{j} \neq 0 \right\}  & \to &  \mathbb C^{n-1} \\
 ( x_1,\dots,x_{n}) & \mapsto &  \left(x_1,\ldots,x_{j-2}, 
 \textstyle{\frac{x_{j-1} x_{j}}{|x_{j}|}},\textstyle{\frac{x_{j}x_{j+1}}{|x_{j}|}},x_{j+2},\ldots,x_{n}\right).
\end{eqnarray*}
Applying equation~\eqref{eq:level} to the
level sets $U_\Delta$ and $U_{\widecheck{\Delta}}$,
$\widetilde\varpi$ maps 
$\{x \in U_\Delta \, | \, x_j\neq 0\}$
onto
$\left\{y\in U_{\widecheck{\Delta}}\ \left| \ \frac{1}{2} \alpha_j |y_{1}|^2+\frac{1}{2}\beta_j|y_{n-1}|^2 >a_j \right. \right\}.$
Moreover, since $v_j=v_{j-1}+v_{j+1}$, we have that if  $x, x' \in \mathbb C^{n}$ satisfy  $x_{j} \neq 0 \neq x'_{j}$,  then $x$ and $x'$
are in the same $K$ orbit if and only if $\widetilde \varpi(x), \widetilde \varpi(x') \in \mathbb C^{n-1}$ are 
in the same $\widecheck{K}$ orbit.
Thus, the map $\varpi$ defined in \eqref{eq:Teq-symplectom} is a 
well-defined, $(S^1)^2$-equivariant bijection that intertwines the moment maps.

Furthermore,
for $[x] \in N$ we have
$$\textstyle{\frac{1}{2}} |x_{j+1}|^2 + \textstyle{\frac{1}{2}} |x_{j-1}|^2 = \textstyle{\frac{1}{2}}|x_{j}|^2 + \varepsilon_j,$$
where $\varepsilon_j=a_j-(a_{j-1}+a_{j+1})$. 
Thus
$\varpi$ has a smooth inverse  given by
$$[ y_1, \dots, y_{n-1}] \mapsto \left[y_1,\dots,y_{j-1},  \big( |y_{j-1}|^2 + |y_{j}|^2 - 2 \varepsilon_j \big)^{\frac{1}{2}}, y_{j}, \dots, y_{n-1}\right].$$
Therefore, $\varpi$ is a  diffeomorphism.
Finally, by a straightfoward computation in coordinates, the pull-back by the map $\widetilde \varpi$ of the standard symplectic form on $\mathbb C^{n-1}$  is the standard symplectic form on $\mathbb C^{n}$ plus a multiple of $\operatorname{d}(|x_{j}|^2 - |x_{j-1}|^2 - |x_{j+1}|^2).$ Therefore, $\varpi$ is a symplectomorphism.
We now restrict to $N \smallsetminus D$; 
applying this blowdown procedure inductively and composing the symplectic blowup maps, we see that $\psi$ has the desired properties.
\end{proof}

\section{Partial flips and proof of Theorem \ref{Thm:main}}\labell{sec:partial}

In this section we construct
an equivariant diffeomorphism between 
two compact, connected,
four-dimensional Hamiltonian $S^1$-manifolds whose oriented dull graphs differ by a single partial flip. 
We give the construction explicitly
in coordinates in triangular neighborhoods of the chains and then use the symplectic blowup maps 
to extend to a global diffeomorphism.   
As a consequence, we deduce Theorem \ref{Thm:main}.

\begin{Definition}\label{def:mirror}
Let $N$ be a corner manifold whose moment image is the polyhedral set 
$$\Delta = \bigcap_{i=1}^n \{ w \in \R^2 \mid \langle w, v_i \rangle \geq a_i \},$$
where $v_i = (\alpha_i, \beta_i)$ for all $i$.
Reflecting the moment map image $\Delta$ of  $N$ 
across the diagonal line $w_1 = w_2$ yields a Delzant polyhedral set $\Delta'$.  In coordinates,
$$
\Delta'=\bigcap_{i=1}^n \left\{ w \in \R^2\ \Big|\ \langle w, v_i' \rangle \geq a'_i \right\},
$$
where
$a'_i=a_{n+1-i}$
and $v_i'=(\alpha_i', \beta_i')=(\beta_{n+1-i},\alpha_{n+1-i})$.
Let $N'$ denote the associated
corner manifold:  $N' \simeq M_{\Delta'}$.  
We call $N'$ the {\bf mirror} of the manifold $N$. 
This procedure is an involution, so $N$ is also the mirror of $N'$ and we
will refer to $N$ and $N'$ as {\bf mirror manifolds}.
\end{Definition}

\begin{Remark}\labell{rem:3'}
If
$U_\varepsilon$ is a triangular $\varepsilon$-neighborhood
of the chain of isotropy spheres $D$ in $N$ for some $\varepsilon > 0$, then $U'_\varepsilon$ is an
$\varepsilon$-neighborhood of the chain of isotropy spheres $D'$ in 
the mirror manifold $N'$; see Figure \ref{fig:corner-flip}. 
\end{Remark}

\begin{figure}[h]
\centering
\includegraphics[width=10cm]{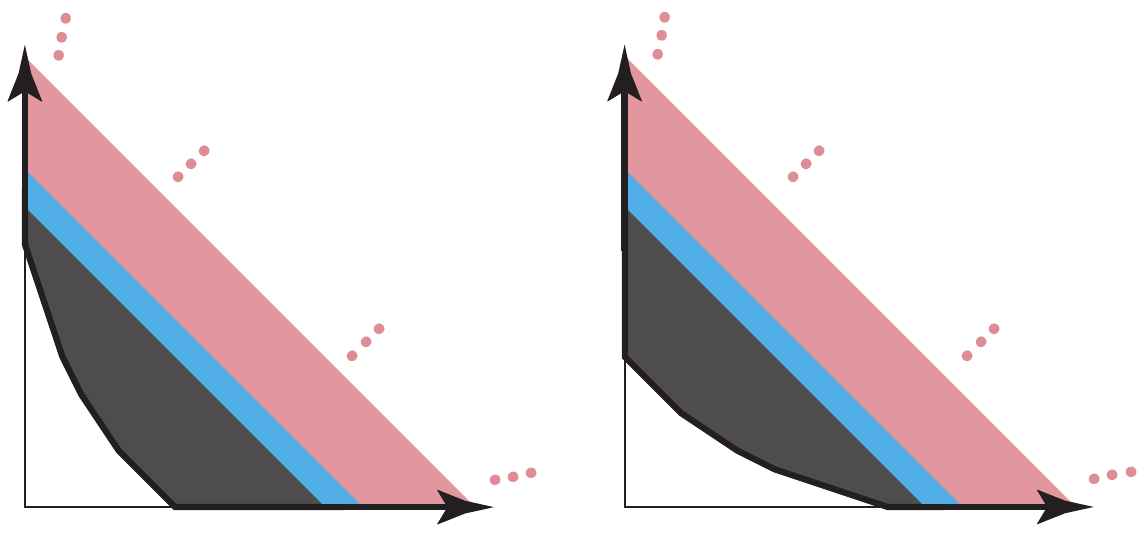} 
\caption[.]{
The shaded regions represent the 
moment map images of triangular neighborhoods of the chains in
a corner manifold and its mirror.}
\label{fig:corner-flip}
 \end{figure}

We start by producing a diffeomorphism between mirror corner manifolds that is $S^1$-equivariant 
with respect to the anti-diagonal action.

\begin{Lemma}\labell{lem:fliptri}
Let $M_{\Delta}$ and $M_{\Delta'}$ be mirror corner manifolds with chains of isotropy spheres $D$ and $D'$, respectively.  The map $\C^n \to \C^n$ given by 
\begin{equation}\labell{eq:1stmap}
(x_1,\ldots,x_n) \mapsto (- \overline{x}_n, \overline{x}_{n-1}, \ldots,\overline{x}_1)
\end{equation}
induces a diffeomorphism $F \colon M_\Delta \to M_{\Delta'}$
that  negates the symplectic form;
 intertwines the moment maps for the diagonal actions on $M_{\Delta}$ and $M_{\Delta'}$;
 is equivariant with respect to the anti-diagonal circle actions on $M_{\Delta}$ and $M_{\Delta'}$;
 and maps the isotropy sphere $S_i \subseteq D$ onto $S_{n+1-i}' \subseteq D'$ for all $1 < i < n$.
Moreover,  \begin{equation}\labell{eq:intertwine}
(\psi' \circ F \circ \psi^{-1})(z_1,z_2) = (-\overline{z}_2, \overline{z}_1)
\end{equation}
for all $(z_1,z_2) \in \psi(M_\Delta \smallsetminus D) \subset \C^2$,
where $\psi \colon M_{\Delta} \smallsetminus D \to \C^2$ and $\psi'\colon M_{\Delta'} \smallsetminus D' \to \C^2$ are the symplectic blowup maps.
\end{Lemma}

\begin{proof}
By definition, the normal vectors $v_i=(\alpha_i,\beta_i)$, $v'_i=(\alpha'_i,\beta'_i)$ and
the constants $a_i$,  $a'_i$ defining $\Delta$ and $\Delta'$ satisfy 
\begin{equation*}
(\alpha_i', \beta_i')=(\beta_{n+1-i},\alpha_{n+1-i}) \mbox{ and } a_i' = a_{n + 1 - i}.
\end{equation*}
By equations \eqref{eq:lie-alg}-\eqref{eq:red}, this implies that the symplectomorphism $\C^n \to \C^n$ given by $(x_1,\dots,x_n) \mapsto (x_n, \dots, x_1)$
takes $\cU_\Delta$ to $\cU_{\Delta'}$ and $K_\Delta$ to $K_{\Delta'}$. Hence it induces a symplectomorphism 
$f \colon M_\Delta \to M_{\Delta'}$.  By equations \eqref{eq:action} and \eqref{eq:mmap},
since the map that sends $(x_1,\dots,x_n) \mapsto (x_n, \dots, x_1)$ is weakly equivariant with respect to the automorphism of $(S^1)^n$ that reverse the coordinates,  the symplectomorphism  $f$ is weakly equivariant with respect to the automorphism 
of $(S^1)^2$ that exchanges the coordinates, 
and $f^*(\Phi_{\Delta'})$ can be obtained from $\Phi_\Delta$ by exchanging the coordinates.
Therefore, by equation \eqref{eq:Si1}, $f(S_i)=S'_{n+1-i}$ for each $i$.
Similarly, the map from $\C^n \to \C^n$ given by $(x_1,\dots,x_n) \mapsto (\overline{x}_1,\dots, \overline{x}_{n-1}, -\overline{x}_n)$ induces a diffeomorphism from $M_\Delta$ to itself that negates the symplectic form, is weakly equivariant with respect to inversion, and preserves the moment map.  Therefore, the map $\C^n \to \C^n$ given in (4.5), which is the composition of these two maps, induces a diffeomorphism $F \colon M_\Delta \to M_{\Delta'}$ with the desired properties.
Finally, we can verify \eqref{eq:intertwine} by a short calculation
using Lemma~\ref{lem:strip}.\end{proof}

Next, we construct an $S^1$-equivariant diffeomorphism between mirror corner manifolds that 
interpolates between (a) the diffeomorphism $F$ defined in Lemma 3.15 in a triangular neighborhood $U_{\alpha}$ 
and (b) a lift of   the identity map in the complement of a triangular neighborhood $U_{\beta}$, where $\beta>\alpha$.
In Figure \ref{fig:corner-flip}, the dark grey regions represent the moment map images of triangular neighborhoods $U_{\alpha}$ and $U'_{\alpha}$ in a corner manifold $N$ and its mirror $N'$; the pink regions represent the moment map images of the complements of $U_{\beta}$ and $U'_{\beta}$; and the blue region in the middle is where the interpolation takes place.

\begin{Lemma}\labell{lem:combined}
Let $N$ and $N'$ be mirror corner manifolds with chains of isotropy spheres $D$ and $D'$, respectively.  
Let  $U_\beta$ be a  triangular neighborhood  of $D$ for some $\beta > 0$.  There is an orientation-preserving  diffeomorphism ${H} \colon N  \to N' $ such that
$$\textstyle (\psi' \circ H \circ \psi^{-1})(z_1,z_2) = (z_1,z_2) \text{ for all } (z_1,z_2) \in \psi(N \smallsetminus U_\beta) \subset \C^2. $$
Moreover, the map $H$ intertwines the moment maps for the diagonal circle actions on $N$ and $N'$; is equivariant with respect to the anti-diagonal circle actions on $N$ and $N'$; and maps the isotropy sphere $S_i \subseteq D$ onto $S'_{n+1-i} \subseteq D'$ for all $1 < i < n$.
Here, $\psi \colon N\smallsetminus D \to \C^2$ and $\psi' \colon N'\smallsetminus D' \to \C^2$ are the symplectic blowup maps.
\end{Lemma}

\begin{proof}
Since  $D$ is contained in $U_\beta$, it is also contained in $U_\alpha$ for some $\alpha < \beta$.  Since $N$ and $N'$ are mirror manifolds,
$D'$ is also contained in $U'_\alpha$. See Remarks \ref{rem:shrink} and \ref{rem:3'}.

Let $F \colon N \to N'$ be the diffeomorphism constructed in  Lemma \ref{lem:fliptri}. Then
\begin{equation}\labell{eq:F}
(\psi' \circ {F} \circ \psi^{-1})(z_1,z_2) = (- \overline{z}_2, \overline{z}_1)
\end{equation}
for all $(z_1,z_2) \in \psi(N \smallsetminus D) \subset \mathbb C^2$.
Moreover, the   diffeomorphism $F$ negates the symplectic form, and so  is orientation-preserving.
The map $F$ also intertwines the moment maps for the diagonal circle actions on $N$ and $N'$.  
Since the triangular neighborhoods $U_\alpha \subset N$ and $U'_\alpha \subset N'$ are each the preimage of $[0,\alpha)$
under the moment map for the respective diagonal circle actions,
  $F$  restricts to a diffeomorphism between $U_\alpha$ and $U'_\alpha.$
Furthermore, the map $F$ is  equivariant with respect 
to the {anti-diagonal} circle actions on $N$ and $N'$.
 Finally, $F$ maps the isotropy sphere $S_i \subseteq D$ onto $S'_{n+1-i} \subseteq D'$ for all $1<i<n$,
and thus  $F$ restricts to a diffeomorphism between 
$U_\alpha \smallsetminus D$ and $U'_\alpha \smallsetminus D'$.

Let $S^1$ act on $\C^2$ by $\lambda \cdot (z_1,z_2) = (\lambda z_1, \lambda^{-1}z_2).$
 The family
    $$\varphi_t \colon \C^2 \to \C^2; \,\,\, \varphi_t(z_1,z_2)=\Big(\cos{\big(\textstyle{\frac{\pi}{2}} t\big)}z_1-\sin {\big(\textstyle{\frac{\pi}{2}} t\big)}\overline{z}_2, \cos{\big(\textstyle{\frac{\pi}{2}} t\big)}z_2+\sin{\big(\textstyle{\frac{\pi}{2} } t\big)}\overline{z}_1\Big)$$
    gives a smooth isotopy of orientation-preserving $S^1$-equivariant diffeomorphisms  that satisfies $\varphi_0=\id$ and $\varphi_1(z_1,z_2)=(- \overline{z}_2, \overline{z}_1)$. Moreover, $(\varphi_t)^{-1}=\varphi_{-t}$ and $\varphi_t$ preserves the standard norm 
    on $\C^2$ for all $t$.
    Define a cut-off function
    $\rho  \colon \R_{\geq0} \to [0,1]$
    such that  $\rho$ equals $1$ on $[0,\alpha]$, and
     equals $0$ on  $[\beta, \infty)$. The map
    \begin{equation*}\labell{eq:iso+cut}
    g \colon \C^2 \to \C^2; \,\,\, g(z_1,z_2)=\varphi_{\rho(|z|)}(z_1,z_2) 
    \end{equation*}
    is an orientation-preserving $S^1$-equivariant diffeomorphism 
   that satisfies, for each point
$z=(z_1,z_2)\in\C^2$,
\begin{equation}\labell{eq:newg}
g(z_1,z_2) = \begin{cases} (-\overline{z}_2,\overline{z}_1) & \text{if } \frac{1}{2} |z|^2 \leq  \alpha \\
(z_1,z_2)  & \text{if } \frac{1}{2}|z|^2  \geq  \beta.
\end{cases}
\end{equation}
 Its inverse is the map $(z_1,z_2) \mapsto \varphi_{-\rho(|z|)}(z_1,z_2)$. 
 Moreover, $g$ preserves the standard norm $|z| = \sqrt{ |z_1|^2 + |z_2|^2}$ on $\C^2$, or equivalently the moment map for the diagonal circle action.

Since $N$ and $N'$ are mirror corner manifolds, the map $(z_1,z_2) \mapsto (-\overline{z}_2,\overline{z}_1)$ takes $\psi(N \smallsetminus D)$ to $\psi'(N' \smallsetminus D')$. Moreover,
because $D \subset U_\alpha$, Lemma~\ref{lem:strip} implies that  every  $z\in\C^2$ with $\frac{1}{2} |z|^2 > \alpha$ is contained in both  $\psi(N \smallsetminus D)$ and $\psi'(N' \smallsetminus D')$.
Therefore,  since $g$ preserves the standard norm on $\mathbb C^2$, the map $g$ induces a diffeomorphism from $\psi(N\smallsetminus D)$ to $\psi'(N'\smallsetminus D')$.  Thus, we can define  a diffeomorphism $G$ from $N \smallsetminus D$ to $N' \smallsetminus D'$ by
\begin{equation}\labell{eq:G}
G = (\psi')^{-1} \circ g \circ \psi.
\end{equation}
Since  $\psi$ is an $(S^1)^2$-equivariant symplectomorphism that intertwines the moment maps, 
the map $G$ is orientation-preserving, intertwines the moment maps for the diagonal circle actions,  and is  equivariant
with respect to the anti-diagonal circle action. 

Finally, ${G}$ coincides  with ${F}$ on $\overline{U_\alpha} \smallsetminus D$, and 
with $(\psi')^{-1} \circ \psi$ on $N \smallsetminus {U}_\beta$,
where $\overline{U_\alpha}$ denotes the topological closure of $U_\alpha$.
Thus we may piece together the maps ${F}|_{U_\alpha}$ and ${G}$ to create the required diffeomorphism $H:N\to N'$.
\end{proof}

\noindent We can now construct the  equivariant
diffeomorphism 
associated to a partial flip in an oriented dull graph.

\begin{Proposition}\labell{prop:partial-flip-diffeo}
Let $(M,\omega,\Phi)$ be compact, connected four-dimensional Hamiltonian 
$S^1$-man\-i\-fold, let   $(\G, \mathcal O)$ be the oriented dull graph of $M$, and let $\mathcal O'$ be
the partial flip of $\mathcal O$ along a free chain $C$.
There exists a compact, connected four-dimensional Hamiltonian $S^1$-manifold $(\widehat M, \widehat \omega, \widehat \Phi)$ and an orientation-preserving, equivariant diffeomorphism from $M$ to $\widehat M$
that induces an isomorphism from $(\G, \mathcal O')$ to the oriented dull graph of $\widehat M$.
\end{Proposition}

\begin{proof}
By Lemma~\ref{lem:new} there exists   a compact, connected, four-dimensional  Hamiltonian $S^1$-manifold $\widetilde{M}$ whose oriented dull graph is  isomorphic to $(\mathcal G, \mathcal O)$,  and an $S^1$-equivariant symplectomorphism from a triangular neighborhood $U_\gamma$ in a corner manifold $N$, equipped with the anti-diagonal circle action, to a neighborhood of the free chain in $\widetilde M$
 corresponding to $C$.
Identify this 
neighborhood in $\widetilde{M}$ with  $U_\gamma$.
Let $N'$ be the mirror manifold to $N$ and let $D$ and $D'$ be the chains of isotropy spheres in $N$ and $N'$, respectively. Because the chain $D$ is  contained in $U_\gamma$, it is also contained in $U_\beta$ for some $0 < \beta < \gamma$. Moreover, the chain $D'$ is contained in $U_\beta' \subset U_\gamma'$. (See Definition~\ref{def:triangular} and Remarks~\ref{rem:shrink} and ~\ref{rem:3'}.) 

By Lemma~\ref{lem:strip}, the    symplectic blowup maps $\psi \colon N \smallsetminus D \to \mathbb C^2$ and $\psi' \colon N' \smallsetminus D' \to \mathbb C^2$ are $(S^1)^2$-equivariant symplectomorphisms onto their image, and intertwine the moment maps.
Hence, they restrict to $(S^1)^2$-equivariant symplectomorphisms from $U_\gamma \smallsetminus \overline U_\beta$ and $U'_\gamma \smallsetminus \overline U'_\beta$, respectively, to
$$ \textstyle \big\{ z \in \mathbb C^2\,\, \big| \,\, \beta< \frac{1}{2}|z|^2 < \gamma \big\}.$$
Thus, the map $(\psi')^{-1} \circ \psi$ restricts to an $(S^1)^2$-equivariant symplectomorphism from $U_\gamma \smallsetminus \overline U_\beta$ to $U'_\gamma \smallsetminus\overline U'_\beta$ that intertwines the moment maps.
Using this isomorphism, we   construct a new Hamiltonian $S^1$-manifold $\widehat{M}$ 
by removing the closed set $\overline{U}_\beta$ from $\widetilde{M}$ and gluing in  $U'_\gamma$, equipped with the anti-diagonal circle action.

By  Lemma~\ref{lem:hef},  there is an orientation-preserving
equivariant diffeomorphism from $M$ to $\widetilde M$ that induces  an isomorphism from
$(\G, \mathcal O)$ to the oriented dull graph of $\widetilde M$.
Use this to identify the dull graph of $\widetilde M$ with $\G$.
By Lemma~\ref{lem:combined}, there
exists an $S^1$-equivariant, orientation-preserving diffeomorphism $H \colon N \to N'$ that  restricts to
$(\psi')^{-1}\circ\psi$ on $N\smallsetminus \overline{U}_\beta$ and intertwines the moment maps for the diagonal circle on $N$ and $N'$,
and so restricts to a diffeomorphism from $U_\gamma$ to $U'_\gamma.$
This gives an orientation-preserving, equivariant diffeomorphism from $\widetilde{M}$ to $\widehat{M}$ that 
induces an isomorphism from $(\G, \mathcal O')$ to the oriented dull graph of $\widehat M.$ 
Composing, we have an orientation-preserving, equivariant diffeomorphism from $M$ to $\widehat M$ 
that induces the 
isomorphism from $(\G, \mathcal O')$ to the oriented dull graph of $\widehat M.$
\end{proof}

We now have all the pieces in place to prove that every isomorphism of dull graphs is
induced by an orientation-preserving, equivariant diffeomorphism.

\begin{proof}[Proof of Theorem \ref{Thm:main}]
Assume that we are given an isomorphism between the  dull graphs $\G$ and $\widetilde \G$ of  two compact,  connected  four-dimensional Hamiltonian $S^1$-manifolds $(M,\omega,\Phi)$ and $(\widetilde{M},\widetilde \omega, \widetilde \Phi)$.
Then there is a sequence of orientations
$\mathcal O_0$, $\mathcal O_1$,
$\dots$, $\mathcal O_n$ on $\G$ such that
  $\mathcal O_0$ is the orientation induced by $M$, the given isomorphism takes $(\G, \mathcal O_n)$ to the orientation on $\widetilde \G$ induced by $\widetilde M$, and  for each $k$,  the orientation  $\mathcal O_k$ 
 is either the opposite to the  orientation  $\mathcal O_{k-1}$ or differs from  $\mathcal O_{k-1}$ by a single partial flip along a free chain.
We now  apply Lemma~\ref{lem:fullflip} and Proposition~\ref{prop:partial-flip-diffeo} 
to the oriented dull graphs $(\G, \mathcal O_0), \dots, (\G, \mathcal O_n)$,
inductively producing a sequence of four-dimensional Hamiltonian $S^1$-manifolds $M_0 = M$, $M_1$, $\dots$, $M_{n-1}$, $M_n$
 and an orientation-preserving equivariant diffeomorphism from  $M$ to $M_{k}$
for all $k$ that induces an orientation-preserving isomorphism from $(\G, \mathcal O_k)$ to  the oriented dull graph associated to $M_k$ for all $k$. 
In particular, there is an orientation-preserving, equivariant diffeomorphism from $M$ to $M_{n}$
that induces an orientation-preserving isomorphism from $(\G, \mathcal O_n)$
to the oriented dull graph associated to $M_n$.
The theorem now follows from Lemma~\ref{lem:hef}.
\end{proof}

\section{Proof of Theorem \ref{thm:strong}}
\label{sec:thm1}

The goal in this section is to prove our main result, Theorem~\ref{thm:strong},
which states that every algebra isomorphism between the equivariant cohomology algebras of 
compact, connected four-dimensional Hamiltonian $S^1$-manifolds 
is induced
by an equivariant diffeomorphism. 
In the special case that the algebra isomorphism is orientation preserving and the equivariant cohomology of 
the manifolds
vanishes in odd degrees, this Theorem -- like Corollary~\ref{cor:main} --  follows nearly immediately from Theorem~\ref{Thm:main} and the results in \cite{HK}.  Therefore, we will start by analyzing the provenance of
the odd-degree cohomology.

\begin{Lemma}\label{cohomin17}
Let $(M,\omega,\Phi)$ be a compact, connected four-dimensional 
Hamiltonian $S^1$-manifold. Let $R$ be the subring of 
$H^*_{S^1}(M;\Z)$ generated by $H^1_{S^1}(M;\Z)$. Then 
\begin{itemize}
\item[(i)] $H_{S^1}^*(M;\Z)$ is generated  as a ring  by $H_{S^1}^{2*}(M;\Z)$ and $H_{S^1}^1(M;\Z)$.
\item[(ii)] If  $H^1_{S^1}(M;\Z)$ is non-zero, 
then the minimum fixed component $M_{\min}$ is a surface  of
positive genus  and the inclusion map $M_{\min} \to M \times_{S^1} ES^1$ induces an isomorphism
from $R$ to $H^*(M_{\min};\Z)$.

\end{itemize}
\end{Lemma}

\begin{proof}
We may assume that  $H^{2i+1}_{S^1}(M;\Z) \neq 0$ for some $i$.
The moment map $\Phi$ is an  equivariantly perfect Morse-Bott function whose 
critical set is the fixed set,
and every fixed component has even index and coindex \cite{AB, GS:convex, Kirwan}.  
Hence, it has  a unique local minimum $M_{\min}$ and
a unique local maximum $M_{\max}$, and so
 $$H^1(M_{\min};\Z) = H^1(M;\Z) = H^1(M_{\max};\Z).$$  Since $M$ is $4$-dimensional,
$M_{\min}$ and $M_{\max}$ have dimension $0$ or $2$, while other fixed points are isolated.
Since $H^{2*+1}_{S^1}(M;\Z) \neq 0$, this implies that $M_{\min}$ and $M_{\max}$
are surfaces of the same positive genus.

As in \cite[\S 4]{TW},  since $H^*(F;\Z)$ is torsion-free for each fixed component $F$,  we
can adapt  the arguments used to prove \cite[Lemma 1.13]{MT} to get the following ``canonical classes" in $H^{2*+1}_{S^1}(M;\Z)$:
\begin{enumerate}
    \item\labell{canon:first}  For each $\alpha \in H^1(M_{\min};\Z)$, there exists $\widetilde \alpha \in H^1_{S^1}(M;\Z)$ 
    such that $\widetilde \alpha|_{M_{\min}} = \alpha$. Here, we use the isomorphism $H^*_{S^1}(F) = H^*(F) \otimes H^*_{S^1}(\pt)$ to
    identify $H^*(F)$ with a subring of 
    $H^*_{S^1}(F)$ for each fixed component $F$.
    \item\labell{canon:second} For each $\beta \in H^1(M_{\max};\Z)$, there exists $\widetilde \beta \in H^3_{S^1}(M;\Z)$
    such that
    $$\widetilde \beta|_{M_{\max}} =\beta \cup e_{S^1}(\nu_{M_{\max}})$$ 
    and $\widetilde \beta$ restricts to zero on every other fixed component.  Here,  $e_{S^1}(\nu_{M_{\max}})$ is the equivariant Euler class of the normal bundle to $M_{\max}$.
    \item There exists  $\varepsilon_{M_{\max}} \in H^2_{S^1}(M_{\max};\Z)$ such that
 $\varepsilon_{M_{\max}}|_{M_{\max}} = e_{S^1}(\nu_{M_{\max}})$ and  $\varepsilon_{M_{\max}}$ restricts to zero on every other fixed component.
\end{enumerate}
Moreover, as a module over  $H^{*}_{S^1}(\pt;\Z)$, the odd-degree cohomology
$H^{2*+1}_{S^1}(M;\Z)$ is generated by the canonical classes of types
\eqref{canon:first} and \eqref{canon:second}; hence,  $H^*_{S^1}(M;\Z)$ is generated by $H^{2*}(M;\Z)$ and the canonical classes of types
\eqref{canon:first} and \eqref{canon:second}.

Applying the same lemma to $-\Phi$, for each $\beta \in H^1(M_{\max};\Z)$ there exists
a class $\beta'$ in $H^1_{S^1}(M;\Z)$ whose restriction to $M_{\max}$ is $\beta$.   
The type \eqref{canon:second} canonical class $\widetilde \beta \in H^3_{S^1}(M;\Z)$
is product of this
class $\beta'$  of degree $1$ with the  class $\varepsilon_{M_{\max}} \in H^2_{S^1}(M;\Z)$  because
 $\varepsilon_{M_{\max}}$ restricts to zero on every non-maximal fixed component and $H^*_{S^1}(M;\Z) \to H^*_{S^1}(M^{S^1};\Z)$ is
injective.  
Claim (i) now follows immediately.

The inclusion map $M_{\min} \hookrightarrow M$ induces a homomorphism
$$
H_{S^1}^*(M;\Z)\to H_{S^1}^*(M_{\min};\Z) \to H^*(M_{\min};\Z).
$$
By the preceding paragraph, this map restricts to an isomorphism from $H_{S^1}^1(M;\Z)$ to $H^1(M_{\min};\Z)$.
 Since $M_{\min}$ is a surface of positive genus, $H^*(M_{\min};\Z)$ is generated by $H^1(M_{\min};\Z)$. Hence, since $R\subset H_{S^1}^*(M;\Z)$ is the subring generated by $H_{S^1}^1(M;\Z)$, the induced homomorphism $R \to H^*(M_{\min};\Z)$
is surjective.
To prove that this homomorphism is injective,  consider a class
$\gamma = \gamma_1 \cup \cdots \cup \gamma_n \in R$ that restricts to zero on $M_{\min}$, where  $n > 1$ and $\gamma_i \in H_{S^1}^1(M;\Z)$ for all $i$.
To prove that $\gamma$ is the zero class, it is enough to prove that its restriction to every  fixed component
vanishes.
Since the odd-dimensional equivariant cohomology of a point is trivial, each $\gamma_i$  restricts to the zero class on every fixed point, and so $\gamma$ restricts to a class in ordinary cohomology on each fixed component.
If $n > 3$, this implies that $\gamma$ vanishes when restricted to any fixed component.
 If $n = 2$,  this implies that $\gamma$ vanishes when restricted to any fixed component except possibly the maximum;
 moreover,  
 $$\int_{M_{\max}} \frac {\gamma|_{M_{\max}}}{e_{S^1}(\nu_{M_{\max}})} = 0 \text{ exactly if } \gamma|_{M_{\max}} = 0.$$
 Because  the integral of a class of degree $2$  over $M$ is
 zero, the ABBV  localization formula \cite{AB,BV} implies that
 $\gamma$ restricts to zero on the maximum as well. This completes the proof of Claim (ii).
\end{proof}

Next, when the odd cohomology is not trivial,  that is,  when the minimal
fixed surface has positive genus,  we need the following stronger version of Theorem~\ref{Thm:main}.

\begin{Lemma}\labell{lem:new eq diff}
Let $M$ and $\widetilde{M}$ be compact, connected, 
four\--dim\-en\-sion\-al Ham\-il\-ton\-ian $S^1$-manifolds with 
minimal fixed components that are symplectic surfaces $\Sigma$ and $\widetilde \Sigma$, respectively, with positive genus.
Let  $\psi$ be an isomorphism from the dull graph of $M$ to the dull graph of $\widetilde M$
that sends $\Sigma$ to $\widetilde \Sigma$, and let 
$\lambda \colon H^{*}(\widetilde \Sigma;\Z) \to H^{*}(\Sigma;\Z)$ 
be  a  ring isomorphism that preserves the orientations induced by the symplectic forms. 
Then there exists an orientation-preserving equivariant diffeomorphism $h\colon M \to \widetilde{M}$ such that $h$ induces 
the given isomorphism $\psi$ and  $(h|_{\Sigma})^* = \lambda.$
\end{Lemma}

\begin{proof}
Let $\psi$ be an isomorphism from the dull graph of $M$ to that of $\widetilde M$.
By Theorem~\ref{Thm:main}, there exists an orientation-preserving  
equivariant diffeomorphism  $f \colon M \to \widetilde M$  that induces 
the given isomorphism $\psi$.

The surface $\Sigma$ is a minimal fixed component.  Hence, if  we use the symplectic orientations on $\Sigma$ and $M$
to orient the normal bundle $\nu(\Sigma)$ to  $\Sigma$, then $S^1$ acts on $\nu(\Sigma)_p$ with weight $+1$ for all $p \in \Sigma$; 
a similar statement applies to $\widetilde \Sigma$.
Therefore,  since $f$ is an orientation preserving equivariant diffeomorphism taking $\Sigma$ to $\widetilde \Sigma$,
the pull-back map in cohomology 
$$(f|_\Sigma)^* \colon H^*(\widetilde \Sigma;\Z) \to H^*(\Sigma;\Z)$$  preserves the symplectic orientations.
By  \cite[Theorems~1.13 and~6.4]{Farb Margalit},  there exists an orientation preserving diffeomorphism $g\colon \widetilde \Sigma \to \widetilde \Sigma$  such that $g^* = \lambda^{-1} \circ  (f|_{\Sigma})^*.$
Then  by Theorem \ref{thm:elucid}, there exists an equivariant symplectomorphism
$\widehat g \colon \widetilde M \to \widetilde M$  that induces the identity map on the dull graph  such that the restriction $\widehat{g}|_{\widetilde{\Sigma}}$
is isotopic to $g$. 
Finally, the map  $h = \widehat{g}^{-1} \circ f  \colon M \to  \widetilde M$ is the desired equivariant diffeomorphism.
\end{proof}

\noindent A final ingredient is to construct explicitly some
orientation-reversing diffeomorphisms.

\begin{Lemma}\label{2reversing}
Let $(M,\omega,\Phi)$ be a compact, connected four-dimensional Hamiltonian $S^1$-manifold with $b_2(M) = 2$.
Then there exists an orientation-reversing $S^1$-equivariant diffeomorphism from $M$ to itself.
\end{Lemma}

\begin{proof}
Assume first that either every fixed point is isolated or $M$ has exactly one fixed surface.
Then by Karshon's classification \cite{karshon}, $M$ is a Hirzebruch surface, and the circle action is
the restriction of a subgroup of the natural $(S^1)^2$ action.  
Specifically, we can identify $M$ with
$$ (\C^2 \smallsetminus \{0\} ) \times_{\C^\times} \CP^1,$$
where $\C^\times$ acts diagonally on $\C^2 \smallsetminus \{0\}$, and acts on $\CP^1$ by $\lambda \cdot [w_1,w_2] = [\lambda^N w_1, w_2]$ for some $N \in \Z$.
The  compact torus $(S^1)^2$ acts on $M$ by 
$$(\alpha, \beta) \cdot [z_1, z_2; w_1, w_2] = [\alpha z_1, z_2; \beta w_1, w_2].$$
Hence the map from $M$ to itself sending $[z_1,z_2;w_1,w_2]$ to $[z_1,z_2; \overline{w_2}, \overline{w_1}]$ is
an orientation-reversing $(S^1)^2$-equivariant diffeomorphism.

Now assume instead that  $M$ has two fixed surfaces.   Then  by Karshon's classification $M$ is a ruled $S^1$-manifold. 
More specifically, we may identify $M$ with $$M := E \times_{\C^\times}  \CP^1, $$
where $E \to \Sigma$ is a principal $\C^\times$ bundle over a compact oriented surface $\Sigma$ and
 $\C^\times$ acts on $\CP^1$ by $\lambda \cdot [w_1, w_2] = [\lambda w_1, w_2]$.
The circle $S^1$ acts on $M$ by $\alpha \cdot [e; w_1, w_2] = [e; \alpha w_1, w_2].$
Hence the map from $M$ to itself sending $[e ;w_1,w_2]$ to $[e ; \overline{w_2}, \overline{w_1}]$ is
an orientation-reversing $(S^1)$-equivariant diffeomorphism.
\end{proof}

We now have all the ingredients needed to prove  Theorem~\ref{thm:strong}.

\begin{proof}[Proof of Theorem \ref{thm:strong}.]
Let  $M$ and $\widetilde{M}$ be compact, connected four-dimensional Hamiltonian $S^1$-manifolds
and let
$$\Lambda \colon H^*_{S^1}(M;\Z) \to H_{S^1}^*(\widetilde M;\Z)$$
be an isomorphism of algebras. We want to show that
there exists an equivariant diffeomorphism 
 from $\widetilde{M}$ to $M$ 
 that induces the algebra isomorphism $\Lambda$.

We begin with the case when $\Lambda$
is orientation-preserving. 
By  \cite[Corollary~7.23]{HK},
there 
exists an isomorphism $\psi$  from the dull graph of $\widetilde M$ to the dull graph of $M$ such that $\Lambda(\varepsilon_{\psi(\widetilde F)}) = \varepsilon_{\widetilde F}$ for each fixed component $\widetilde F \subset  \widetilde M^{S^1}$, where the classes $\varepsilon_{\widetilde F}$ are defined as follows:
Given a fixed component $\widetilde F \subset  \widetilde M^{S^1} \! \! ,$
let $\varepsilon_{\widetilde F}$ denote the class 
whose restriction to $\widetilde F$ is the equivariant Euler class $e_{S^1}(\nu_{\widetilde F})$ of the normal bundle to $\widetilde F$, and whose restriction to every other fixed component is $0$.
We define $\varepsilon_{F}$ for a fixed component $F \subset  M^{S^1}$ analogously.
Here, if $\widetilde F = \{p\}$ is an isolated fixed point, we take the natural (symplectic) orientation on $\nu(\widetilde F) = T_p \widetilde M$, while if  $\widetilde F$ is a surface, we orient the normal bundle $\nu(\widetilde F)$ to $\widetilde F$ so that $S^1$ acts on $\nu(\widetilde F)_p$ with weight $+1$ for all $p \in \widetilde F$. 
Moreover, if the minimal fixed component of $\widetilde M$ is a surface $\widetilde \Sigma$ then
 we may assume that $\psi(\widetilde \Sigma)$ is a minimal fixed 
surface, by replacing $\omega$ by $-\omega$ if necessary.

Let's now also assume that $H^{1}_{S^1}(M;\Z) = 0$.  By Theorem~\ref{Thm:main} there exists an orientation-preserving  $S^1$-equivariant diffeomorphism $h \colon \widetilde M \to M$ \  inducing the isomorphism $\psi$ of dull graphs; in particular, $h(\widetilde F) = \psi(\widetilde F)$ for all fixed components $\widetilde F \subset \widetilde{M}^{S^1}$. 
 Since $h$ is an orientation-preserving  equivariant diffeomorphism, 
 this implies that
$$h^*(\varepsilon_{\psi(\widetilde F)}) =  \varepsilon_{\widetilde F} =  \Lambda(\varepsilon_{\psi(\widetilde F)})$$
for all fixed components $\widetilde F \subset \widetilde M^{S^1}$.  Hence,  by the uniqueness claim in the first statement in \cite[Theorem~7.22]{HK}, we must have
\begin{equation*}\labell{eq:even equal}
h^{*}\big|_{H_{S^1}^{2*}(M;\Z)}=\Lambda\big|_{H_{S^1}^{2*}(M;\Z)}.
\end{equation*}
Therefore, since  $H^{2*+1}(M;\Z) =0$ by Lemma~\ref{cohomin17}(i),
we may conclude that $h^* = \Lambda$.

Now suppose instead that $H^{1}_{S^1}(M;\Z) \neq 0$, and let
$$R \subset H^{*}_{S^1}(M;\Z) \quad  \mbox{and} \quad\widetilde R \subset H^{*}_{S^1}(\widetilde{M};\Z)$$ 
be the subrings 
generated by elements of degree $1$. Then, by 
Lemma~\ref{cohomin17}(ii), the minimal fixed components of $M$ and $\widetilde M$ are
positive genus surfaces $\Sigma$ and $\widetilde \Sigma$, respectively.  
Moreover,   the restriction maps 
$$H_{S^1}^{*}(M;\Z) \to H^*(\Sigma;\Z) \quad
\mbox{and} \quad  H_{S^1}^{*}(\widetilde{M};\Z) \to H^*(\widetilde{\Sigma};\Z)
$$ 
restrict to  isomorphisms between $R$ and $H^*(\Sigma;\Z)$
and between $\widetilde{R}$ and $H^*(\widetilde{\Sigma};\Z)$, respectively.
Since
$\Lambda$ restricts to an isomorphism from  $R$ to $\widetilde{R}$, this implies that 
$\Lambda$ induces an isomorphism in ordinary cohomology  $$\lambda \colon H^*(\Sigma;\Z) \stackrel{\simeq}{\longrightarrow} H^*(\widetilde \Sigma;\Z)$$ such that 
$$\lambda( \alpha|_{\Sigma}) = \Lambda(\alpha)|_{\widetilde \Sigma}$$
for all $\alpha \in R \subset H^*_{S^1}(M;\Z)$.
Moreover,  since $\Lambda$  is orientation-preserving and 
$\Lambda(\varepsilon_{ \Sigma})=\varepsilon_{\widetilde \Sigma}$, 
\begin{equation}\labell{eq:integral of euler}
\int_M \varepsilon_\Sigma\, \alpha = \int_{\widetilde M} \varepsilon_{\widetilde \Sigma}\,\Lambda(\alpha)
\end{equation}
for all $\alpha \in H_{S^1}^2(M;\Z)$. 
Since $\Sigma$ is the minimal fixed component of $M$,  the symplectic orientation on $\Sigma$ and the orientation on $\nu(\Sigma)$ described in the first paragraph  combine
to give the natural (symplectic) orientation on $TM|_\Sigma$. A similar claim applies to $\widetilde \Sigma \subset \widetilde M.$
Hence, by \eqref{eq:integral of euler}, the definition of $\varepsilon_\Sigma$, and the Atiyah-Bott and Berline-Vergne localization theorem,
$$
\int_{\Sigma} \alpha|_{\Sigma} = \int_{\widetilde \Sigma} \Lambda(\alpha)|_{\widetilde \Sigma}
$$
for all $\alpha \in H^2_{S^1}(M;\Z).$
Since the restriction map from $R \subset H^*_{S^1}(M;\Z)$ to $H^*(\Sigma,\Z)$ is surjective, this implies that 
$$
\int_\Sigma \beta = \int_{\widetilde \Sigma} \lambda(\beta)
$$ 
for any $\beta \in H^2(\Sigma;\Z)$; in other words, $\lambda$ is orientation preserving.
Hence, by Lemma~\ref{lem:new eq diff}, there exists an
 orientation-preserving equivariant diffeomorphism $h \colon \widetilde{M} \to {M}$  inducing  
the isomorphism of dull graphs $\psi$ and such that $(h|_{\widetilde{\Sigma}})^* = \lambda.$
As in the previous paragraph, we can now use \cite[Theorem~7.22]{HK}
to deduce from the fact that 
the equivariant diffeomorphism $h$ induces the isomorphism  $\psi$
that 
 $h^{*}\big|_{H_{S^1}^{2*}(M)}=\Lambda\big|_{H_{S^1}^{2*}(M)}$. 
 Since $(h|_{\widetilde{\Sigma}})^* = \lambda$,  Lemma~\ref{cohomin17}(ii) implies that  $h^{*}\big|_{H_{S^1}^{1}(M)}=\Lambda\big|_{H_{S^1}^{1}(M)}$.
 By Lemma~\ref{cohomin17}(i), together these imply  that $h^* = \Lambda.$

Finally, we turn to the case that $\Lambda$ is orientation-reversing. By \cite[Proposition 7.26]{HK}, this implies that $\dim H^2(M;\Z) = \dim H^2(\widetilde{M};\Z) = 2$. Therefore, by Lemma~\ref{2reversing}, 
 there is an orientation-reversing equivariant diffeomorphism $g$ from $M$ to itself.   By definition,  $g^* \circ \Lambda$ is orientation-preserving. So by the previous paragraphs there exists an  equivariant diffeomorphism $h \colon \widetilde{M} \to {M}$ such that $h^* = g^* \circ \Lambda$.  Hence, $(g^{-1} \circ h)^*   = \Lambda$. 
\end{proof}

\end{document}